\documentclass[reqno,11pt]{amsart}

\usepackage{amsmath,amsthm,amssymb,epsfig}
\usepackage{tikz}
\usepackage{float}
\usepackage{mathtools}
\usepackage{subcaption}
\usepackage{algorithm}
\usepackage{algorithmic}
\usepackage{MnSymbol}
\usepackage{graphicx,xcolor}
\usepackage{comment}

\usetikzlibrary{calc}

\theoremstyle{definition}
\newtheorem{theorem}{Theorem}
\newtheorem{lemma}[theorem]{Lemma}

\newtheorem{corollary}[theorem]{Corollary}

\usepackage[left=3.5cm,top=3.5cm,right=3.5cm,bottom=3.5cm]{geometry}

\DeclareMathOperator{\prox}{\text{prox}}

\renewcommand{\v}{\textup{\textsf{v}}}
\newcommand{\contaminated}{S}

\newcommand{\graphminus}{-}
\newcommand{\hindex}{H}

\title{The one-visibility Localization game}
\author[A.\ Bonato]{Anthony Bonato}
\author[T.G.\ Marbach]{Trent G.\ Marbach}
\author[M.\ Molnar]{Michael Molnar}
\author[JD Nir]{JD Nir}
\address[A1,A2,A3,A4]{Toronto Metropolitan University, Toronto, Canada}

\email[A1]{(A1) abonato@torontomu.ca}
\email[A2]{(A2) trent.marbach@torontomu.ca}
\email[A3]{(A3) michael.molnar@torontomu.ca}
\email[A4]{(A4) jd.nir@torontomu.ca}

\begin{document}


\maketitle
\begin{abstract}
We introduce a variant of the Localization game in which the cops only have visibility one, along with the corresponding optimization parameter, the one-visibility localization number $\zeta_1$. By developing lower bounds using isoperimetric inequalities, we give upper and lower bounds for $\zeta_1$ on $k$-ary trees with $k\ge 2$ that differ by a multiplicative constant, showing that the parameter is unbounded on $k$-ary trees. We provide a $O(\sqrt{n})$ bound for $K_h$-minor free graphs of order $n$, and we show Cartesian grids meet this bound by determining their one-visibility localization number up to four values. We present upper bounds on $\zeta_1$ using pathwidth and the domination number and give upper bounds on trees via their depth and order. We conclude with open problems.
\end{abstract}

\section{Introduction}

Pursuit-evasion games, such as the Localization game and the Cops and Robber game, are combinatorial models for detecting or neutralizing
an adversary’s activity on a graph. In such models, pursuers attempt to capture an evader loose on the vertices of a graph. How the players move and the rules of capture depend on which variant is studied. Such games are motivated by foundational topics in computer science, discrete mathematics, and artificial intelligence, such as robotics and network security. For surveys of pursuit-evasion games, see the books \cite{bonato1,bonato2}; see Chapter~5 of \cite{bonato2} for more on the Localization game.

Among the many variants of the game of Cops and Robbers, one theme is to limit the visibility of the robber. For a nonnegative integer $k,$ in $k$-visibility Cops and Robbers, the robber is visible to the cops only when a cop is distance at most $k$. The case when $k=0$ has been studied~\cite{der1,der2,Tosik1985}, as has the case when $k=1$~\cite{bty1, bty2, bty3}, and a recent paper covers the cases $k\geq1$~\cite{kvis_CR_Clarke2020}.

The Localization game was first introduced for one cop by Seager~\cite{car,seager1}. The game in the present form was first considered in the paper \cite{car}, and subsequently studied in several papers such as~\cite{BBHMP,BHM,BHM1,BK,Bosek2018,nisse2,BDELM}.
We consider a novel analogue of one-visibility Cops and Robbers in the setting of the Localization game. In the \emph{one-visibility Localization game}, there are two players playing on a graph, with one player controlling a set of $k$ \emph{cops}, where $k$ is a positive integer, and the second controlling a single \emph{robber}. The game is played over a sequence of discrete time-steps; a \emph{round} of the game is a move by the cops and the subsequent move by the robber. The robber occupies a vertex of the graph, and when the robber is ready to move during a round, they may move to a neighboring vertex or remain on their current vertex. A move for the cops is a placement of cops on a set of vertices (note that the cops are not limited to moving to neighboring vertices). The players move on alternate time-steps, with the robber going first. In each round, the cops $C_1,C_2,\ldots ,C_k$ occupy a set of vertices $u_1, u_2, \dots , u_k$ and each cop sends out a \emph{cop probe} $d_i$, where $1\le i \le k$. If a cop $C_i$ is on the vertex of the robber, then $d_i=0$. If the cop $C_i$ is adjacent to the robber, then $d_i=1$. In all other cases, the cop probe returns no information, and we set $d_i=\ast$. Hence, in each round, the cops determine a \emph{distance vector} $D=(d_1, d_2, \dots ,d_k)$ of cop probes. Relative to the cops' position, there may be more than one vertex $x$ with the same distance vector. We refer to such a vertex $x$ as a \emph{candidate of} $D$ or simply a \emph{candidate}.  The cops win if they have a strategy to determine, after a finite number of rounds, a unique candidate, at which time we say that the cops {\em capture} the robber. We assume the robber is \emph{omniscient}, in the sense that they know the entire strategy for the cops. If the robber evades capture, then the robber wins. For a graph $G$, define the \emph{one-visibility localization number} of $G$, written $\zeta_1(G)$, to be the least positive integer $k$ for which $k$ cops have a winning strategy in the one-visibility Localization game. The standard Localization game is played in the same way, except that  each cops' probe returns $d_i$ as  the distance between this cop and the robber. The localization number is the minimum number of cops required for this game and is denoted $\zeta(G)$. 

For a graph $G$ of order $n$, $\zeta(G) \le \zeta_1(G)$, as a winning cop strategy in one-visibility localization will also be winning in the Localization game. Further, $\zeta_1(G) \le n - 1$. If $G$ has diameter at most 2, then a probe of $\ast$ by a one-visibility cop can only represent a distance of 2, and so $\zeta(G) = \zeta_1(G)$. We may define $\zeta_j$ for all integers $j\ge 0$ in an analogous fashion, although we will only consider the case $j=1$ in this paper. A recent work~\cite{BL22} introduced the so-called \emph{zero-visibility search game}, which is equivalent to $\zeta_j$ in the case when $j=0.$

We illustrate briefly how the parameters $\zeta(G)$ and $\zeta_1(G)$ may differ. \emph{Spiders} are trees with exactly one vertex of degree at least 3.  This vertex is referred to as the \emph{head}, and the paths from the head to the leaves, not including the head, are referred to as \emph{arms}. Let $G$ be the spider consisting of head vertex $r$ and three arms of length three. It is straightforward to see that $\zeta(G) = 1;$ however, $\zeta_1(G) = 2$. To see that $\zeta_1(G) >1$, suppose one cop plays. The robber chooses one of the neighbors of $r$ on an arm the cop will not probe first and passes until the cop is about to probe on that arm. When the cop moves to the robber's arm, the robber moves to $r$. Anticipating the next cop probe, they move to a neighbor of $r$ on an arm that will not be probed in the next round and the process repeats. To see that $\zeta_1(G) \le 2$, have one cop probe $r$ in every round, and the other cop scans each of the arms until the robber is captured. 

The paper is organized as follows. We begin in Section~\ref{spr} by considering a relaxation of the one-visibility Localization game to the one-proximity game, where the robber is captured if they occupy a neighbor of the cop. We consider bounds on $\zeta_1$ in terms of the corresponding one-proximity number $\prox_1(G).$ In Section~\ref{sec:gen_bounds}, we give several techniques for bounding $\zeta_1(G)$ and $\prox_1(G)$ for general graphs $G$. Upper bounds are given using pathwidth and the domination number, and are found for certain minor-free graphs. Lower bounds are derived by using isoperimetric inequalities and a new graph parameter we call the $h$-index. In Section~\ref{sec:trees}, we show that $\zeta_1$ and $\prox_1$ differ on trees by at most 1. We derive upper bounds on trees via their depth and order and give lower bounds on $k$-ary trees with $k\ge 2$ using their isoperimetric peaks. One consequence of these results is that the one-visibility number is unbounded on the family of $k$-ary trees; this contrasts significantly from the Localization game, where trees have localization number at most 2. We consider Cartesian grid graphs in Section~\ref{sec:grids} and derive bounds there that differ by four values. We conclude with further directions and open problems.

All graphs we consider are finite, undirected, reflexive, and do not contain multiple edges. We only consider connected graphs, unless otherwise stated. 
The set of vertices that share an edge with $x$ is denoted $N(x)$, and we refer to vertices in $N(x)$ as \emph{neighbors} of $x$. 
Although our graphs are reflexive, we insist that $x\notin N(x).$ 
We define $N[x]=N(x) \cup \{x\}$. 
For a set $S$ of vertices, $N[S]=\bigcup_{u \in S} N[u]$. 
For a graph $G,$ let $\Delta(G)$ b  e the maximum degree of a vertex in $G$. For further background on graph theory, see \cite{west}. 

\section{The one-proximity game}\label{spr}

Before we present results on the one-visibility Localization game, we give a simpler version that will prove useful for bounding $\zeta_1(G)$. In the \emph{one-proximity game}, play is defined as in the one-visibility Localization game, except that the cops win immediately if any probe returns a distance other than $\ast$. We call the corresponding graph parameter the \emph{one-proximity number}, written as $\prox_1(G)$.  This game corresponds to the probes returning perfect information about the neighborhood of a vertex, rather than merely whether or not the robber is adjacent to the probed vertex. Note that $\prox_1$ is the analogue of the \emph{one-visibility seeing cop-number} $c_1'$, where the robber is captured if they are in the neighborhood of a cop; see \cite{kvis_CR_Clarke2020}.

Observe that $\prox_1(G) \le \zeta_1(G),$ as the one-proximity game cops can use the same strategy as the Localization game cops until the final round when, having located the robber, the one-proximity game cop probes the robber's last known location and must be within distance one from the robber.  As noted for $k$-visibility Cops and Robber in \cite{kvis_CR_Clarke2020}, seeing the robber for the first time could be much more resource intensive than the subsequent capture. The extra expense cannot be too large, however. 

\begin{theorem}\label{thm:prox_zeta_Delta} For every graph $G$, we have 
$$\zeta_1(G) \leq \Delta(G) \prox_1(G).$$
\end{theorem} 
\begin{proof}
    Suppose that when $\prox_1(G)$ cops play the one-proximity game, and that if these cops move on the vertices $V_t$ in round $t$, then the cops win. 
    We play with $\Delta(G) \prox_1(G)$ cops in the one-visibility Localization game. In round $t$, for each $u\in V_t$, a cop is placed on $u$ and on $\Delta(G)-1$ of the at most $\Delta(G)$ vertices in $N(u)$, chosen arbitrarily. 
    We know that in some round $t'$, there is a $v\in V_{t'}$ such that the robber is in $N[v]$. (This was the requirement for the cops to win independent of the robber strategy in the one-proximity game.)
    Before round $t'$, every cop in the one-proximity game received a distance of $\ast$, so the cops in the one-visibility Localization game are playing with no less information. In round $t'$, we have either a cop on the same vertex as the robber, or the robber is on the unique vertex in $N(v)$ that does not contain a cop. In the latter case, the robber's exact location is now known, so they are captured. 
\end{proof}

Theorem \ref{thm:prox_zeta_Delta} is tight on the complete graphs. 
We can say more if $\prox_1(G)$ is large compared to the maximum degree of $G$. We do not claim the bound on $\prox_1$ in the hypothesis of the following theorem is optimal. 
 
\begin{theorem}\label{tpb}
If $G$ is a graph and $\prox_1(G) \ge \Delta(G)^2$, then $\zeta_1(G) = \prox_1(G)$.
\end{theorem} 
\begin{proof}
The $\prox_1(G)$-many cops play the one-visibility Localization game, following a winning strategy from the one-proximity game. At some point, since the strategy is winning, at least one probe returns a distance of $1$ on vertex $v$. 
On the robber's move, the robber moves to a vertex of distance $0$, $1$, or $2$ from $v$. 

During the cops' next move, a cop is placed on each vertex of distance $1$ or $2$ from $v$, which requires at most $\Delta(G)^2$ cops. 
Either a cop on some vertex $u$ probes $0$ and the robber is found on $u$, or no cop probes $0$ and the robber is found on $v$. 
\end{proof}

One benefit of considering the one-proximity game instead of the one-visibility Localization game is the success or failure of the cops strategy is independent of the robber's strategy.  Let $G$ be a graph, and for $t \ge 1$, let $V_t$ denote the set of vertices probed by the cops in round $t$. 
Define $\contaminated_t=\contaminated_t(G,\{V_1, \ldots, V_t\})$ to be the set of vertices on which the robber may reside immediately after the cops' $t$th move without having been captured. 

Consider the following three properties: 
\begin{enumerate}
\item the robber can start on any vertex; 
\item if the robber was on some vertex $v$ before the robber's $(t+1)$th move, then they can move to any vertex in $N[v]$ on their $(t+1)$th move; and
\item the robber is captured on the cops' $(t+1)$th move if they are in a vertex of $\bigcup_{v \in V_{t+1}} N[v].$
\end{enumerate}

The following theorem translates these properties to statements about $\contaminated_t$. For a set of vertices $S$, let $\delta(S)$ be the set of vertices not in $S$ that are adjacent to some vertex in $S$. To avoid conflicting notation, we do not use $\delta$ to denote the minimum degree of a graph.

\begin{lemma} \label{lem:oneVisToSets}
When playing the one-proximity game on a graph $G$ where the cops play on $V_t$ in round $t$, the robber can be on a vertex $u$ if and only if $u\in S_t$, where
\begin{enumerate}
    \item $\contaminated_1=V(G)$; and 
    \item $\contaminated_{t+1} = (\contaminated_t \cup \delta (\contaminated_t)) \setminus (\bigcup_{v \in V_{t+1}} N[v])$.
\end{enumerate}
\end{lemma}
\begin{proof}
The robber may start on any vertex, so $\contaminated_1=V(G)$. Immediately before the robber's $(t+1)$th move, the robber may be on a vertex $v$ if and only if $v\in \contaminated_t$. The robber uses their $(t+1)$th move to occupy some vertex in $N[v]$. Thus, the robber can be on vertex $v$ if and only if $v \in (\contaminated_t \cup \delta (\contaminated_t))$ after the robber's $(t+1)$th move.

The robber will be captured on the cops' $(t+1)$th move if and only if it is in a vertex of $\bigcup_{v \in V_{t+1}} N[v]$. Therefore, the robber remains uncaptured after the cop's $(t+1)$th move if and only if it is on a vertex in $ (\contaminated_t \cup \delta( \contaminated_t)) \setminus (\bigcup_{v \in V_{t+1}} N[v])$. This completes the proof. 
\end{proof}
  
There are a variety of different terminologies for the sets $\contaminated_t$. These can be called the \emph{robber territory}, 
or the set of \emph{contaminated} vertices. We use the term contaminated, denoting these vertices as red in the figures. 
The vertices not in $\contaminated_t$ are usually called either \emph{clean} or \emph{cleared}. 
We use the term cleared and denote these vertices as white in any figures. 
A set of vertices is contaminated (respectively, cleared) if all of its contained vertices are contaminated (respectively, cleared). 
We say a cleared set $S$ is \emph{fully cleared} when the robber can never return to recontaminate the vertices of $S$ under the given cop strategy. 

Lemma~\ref{lem:oneVisToSets} is the one-visibility Localization game equivalent of Proposition~10 of \cite{BL22}, which is a result about the zero-visibility Localization game.
As a result, we can treat play in the one-proximity game as a single-player game where the cops clear vertices on their move and the contamination spreads between the cops' moves. This will often be easier to analyze because the robber strategy is no longer necessary.

\section{Bounds on $\zeta_1$} \label{sec:gen_bounds}

In the present section, we focus on several bounds for $\zeta_1,$ including upper bounds using pathwidth, the domination number, and one using properties of certain minor-free graphs. We finish by giving lower bounds using isoperimetric inequalities.

\subsection{Upper bounds}

We begin with an upper bound using pathwidth. In \cite{Bosek2018}, the localization number of a graph is bounded above by the graph's pathwidth. An analogous result holds for the one-localization number. 

Given a graph $G$, a \emph{path-decomposition} of $G$ is a pair $(X,P)$, where the set $X = \{B_1, B_2, \dots, B_n\}$ consists of subsets of $V(G)$ called $bags$, and $P$ is a path whose vertices are the bags $B_i$, satisfying the following properties: 
\begin{enumerate}
    \item $V(G) = \bigcup_{i=1}^{n}B_i$; 
    \item for every edge $(u,v)\in E(G)$, there exists a bag that contains both $u$ and $v$; and
    \item for all $1 \le i \le k \le j \le n$, $B_i \cap B_j \subseteq B_k$.
\end{enumerate}
The \emph{width} of the path-decomposition is the cardinality of its largest bag minus 1, and the \emph{pathwidth} of the graph $G$, denoted $\mathrm{pw}(G)$, is the minimum width among all possible path-decompositions of $G$. While the proof of the following theorem is analogous to the proof bounding the localization number by pathwidth given in \cite{Bosek2018}, we include it for completeness.

\begin{theorem}\label{pw}
For any graph $G$, $\zeta_1(G) \le \mathrm{pw}(G)$.
\end{theorem}
\begin{proof}
Assume $G$ has at least two vertices and let $P$ be a path-decomposition of $G$. Without loss of generality, linearly order the bags $B_1, B_2, \dots, B_k$ from left to right. For all $1 \le i \le k$ we assume $B_i \setminus B_{i+1}$ is nonempty; otherwise, $B_i$ can be eliminated from the path-decomposition. Furthermore, for every $u \in B_i \setminus B_{i+1}$, we assume $u$ has a neighbor in $B_i$. If this were not the case, then we remove $u$ from bag $B_i$ without changing the path-decomposition.

For each $1 \le i < k$, let $u_i$ be a fixed vertex in $B_i \setminus B_{i+1}$ and let $v_i$ be a neighbor of $u_i$ in $B_i$. Also let $u_k$ be a vertex in $B_k \setminus B_{k-1}$ and $v_k$ be a neighbor of $u_k$ in $B_k$. Sequentially, for $i = 1, 2, \dots, k$, the cops probe each vertex of $B_i \setminus v_i$. 

Starting with $B_1$, which is a leaf of $P$, cops probe $B_1 \setminus v_1$. Suppose the robber is in $B_1$. If they are in $B_1 \setminus v_1$, then they are captured since a cop will probe 0. If the robber is on $v_1$, the cop at $u_1$ probes 1. Since $u_1$ must have a neighbor in $B_1$ and no cop has probed 0, the robber is captured at $v_1$. 
In this way, we can ensure the robber is not in $B_1$. We then proceed inductively, probing the vertices in $B_j \setminus v_j$, for $j > 1$, to ensure the robber is not in $B_i$, with $i \le j$. The robber is forced to move into $B_k$ where they will be captured.
\end{proof}

The bound in Theorem~\ref{pw} is tight for complete graphs $K_n$, as $\mathrm{pw}(K_n) = \zeta_1(K_n) = n - 1$. We note that proof of Theorem~\ref{pw} also holds for the zero-visibility Localization game, except that we must place a cop on all vertices of a bag as we sequentially probe the bags.  Therefore, we also have the following. 

\begin{lemma}
For any graph $G$, $\zeta_0(G) \le \mathrm{pw}(G)+1$.
\end{lemma}

For graphs which are $C_4$-free (that is, do not contain the 4-cycle as a subgraph), we have the following bound for the one-visibility localization number in terms of two common graph parameters, including the domination number, written $\gamma (G).$ 

\begin{theorem}
If $G$ is $C_4$-free, then $\zeta_1(G) \le \gamma(G) + \Delta(G).$
\end{theorem}
\begin{proof}
Let $S = \{v_1, v_2, \dots, v_{\gamma(G)}\}$ be a dominating set of $G$ and have $\gamma(G)$-many cops probe each vertex of $S$ in each round. On the first probe, either some cop probes 0 and the robber is captured, or the robber is on some vertex $u \in V(G)\setminus S$, and there is at least one cop who probes 1. Say this is the cop at $v_1$. 

In the next round, we will use the additional $\Delta(G)$-many cops to probe each neighbor of $v_1$, and so the robber must move to avoid capture. The robber moves to $w \in V(G)\setminus S$, where $w$ is not a neighbor of $v_1$. On the next round of probes, since $S$ is a dominating set, there will be some other cop, say the one at $v_2$, who probes 1, while the cop at $u$ also probes 1. If $v_2$ were adjacent to a second neighbor of $u$, then $G$ would contain a 4-cycle. Therefore, the cops can uniquely determine the robber's location to be at $w$.
\end{proof}

Our next result provides an upper bound on $\zeta_1$ for a large family of graphs. A graph $H$ formed from $G$ by first taking a subgraph and then contracting some of the remaining edges is said to be a \emph{minor} of $G.$ The family of $K_h$-minor free graphs includes planar graphs in the case $h=5$. The following separator theorem for $K_h$-minor-free graphs will be useful for bounding the one-localization number.

\begin{theorem}[\cite{alon}]\label{ast}
If $h \ge 1$ is a fixed integer and $G$ is a $K_h$-minor-free graph of order $n$, then there are sets of vertices $A$, $B,$ and $C$ so that no vertex in $A$ is adjacent with a vertex in $B$, neither $A$ nor $B$ contains more than $2/3n$ vertices, and $C$ contains no more than $h^{3/2}\sqrt{n}$ vertices.
\end{theorem}

We refer to the set $C$ in Theorem~\ref{ast} as a \emph{separator}, and the sets $A$ and $B$ of order at most $2/3n$ as \emph{parts}. The following theorem gives an upper bound on the $\zeta_1$ number for several classes of graphs, including planar graphs. The proof uses a divide-and-conquer approach. Later in the paper, we will give bounds $\Omega(\sqrt{n})$ in square grids. We do not attempt to optimize constants in the upper bound.

\begin{theorem}\label{plan} If $h > 3$ and $n$ are integers, and $G$ is a $K_h$-minor free graph of order $n$, then $\zeta_1(G) = O(\sqrt{n}).$ In particular, if $G$ is planar of order $n$, then $\zeta_1(G) = O(\sqrt{n}).$
\end{theorem}

\begin{proof} 
Define the function $f(m)=h^{3/2}\sqrt{m}\left(\frac{1}{1-\sqrt{2/3}}\right)+ \sqrt{n}$.  
We write $P(m)$ to be the statement that for each $K_h$-minor-free graph $G$ of order $m$, there exists a strategy using at most $f(m)$ cops to capture the robber on $G$. We apply induction, assuming that $P(m)$ is true for $1 \leq m \leq n-1$, and show that $P(n)$ holds. Once that is established, the proof of the theorem follows.

The base cases are when $1 \leq m \leq \sqrt{n}$. In any such graph, we can place a cop on every vertex to capture the robber in one round. This uses at most $\sqrt{n} \leq f(m)$ cops, and so $P(m)$ is true in these cases. 

For the inductive step, for a graph $G$ of order $n$, we apply Theorem~\ref{ast} to give a separator $C$ of $G$
 with $|C| \le h^{3/2}\sqrt{n}$, with parts $A$ and $B$ of cardinalities at most $2n/3.$ 
We note that $A$ and $B$ are not necessarily connected; however, they are both $K_h$-minor-free graphs. 

The cops employ a strategy in two phases. During both phases (and hence, in all rounds), at most $h^{3/2}\sqrt{n}$ cops are played on $C$ so that each vertex in $C$ contains a cop; we label this set of cops by $X$. As such, the robber cannot move between $A$ and $B$ without being captured on a cop move.  

We know by the inductive hypothesis that a strategy exists to capture the robber on $A$ using at most $f(|A|)$ cops. Therefore in the first phase, we play this strategy on $A$ using the cops while also playing the cops in $X$ on $C$. After this, the robber will be captured if it was ever on a vertex in $A$ or $C$, and so we may assume that the robber is now in $B$ after the cops' last move.

There similarly exists a strategy to capture the robber on $B$ using at most $f(|B|)$ cops, and in the second phase, we play this strategy on $B$ while also playing the cops in $X$ on $C$. After this process, the robber will be captured either on a vertex of $B$ or $C$. 

Assuming without loss of generality that $|A| \geq |B|$, we used at most 
\begin{eqnarray*}
\max(h^{3/2}\sqrt{n}  +f(|A|), h^{3/2}\sqrt{n}  +f(|B|)) &\leq & h^{3/2}\sqrt{n}+f(|A|) \\
&\leq & h^{3/2}\sqrt{n}  +h^{3/2}\sqrt{2n/3} \left( \frac{1}{1-\sqrt{2/3}}\right)+ \sqrt{n} \\
&=& h^{3/2}\sqrt{n}\left(\frac{1}{1-\sqrt{2/3}}\right)+ \sqrt{n}
\end{eqnarray*}
cops to capture the robber on $G$. In the first inequality, we used the fact that $|f(B)| \le |f(A)|,$ while the second follows by inductive hypothesis. Hence, $P(n)$ holds, and the proof follows. \end{proof}

Interestingly, we show that the bound in Theorem~\ref{plan} is tight in the sense that there exist planar graphs $G$ (in particular, Cartesian grid graphs) with $\zeta_1(G) = \sqrt{|V(G)|}+O(1)$.

\subsection{Lower bounds from isoperimetric inequalities}

The \emph{isoperimetric} problem of a graph $G$ asks for the minimum cardinality of the boundary of a set of vertices, given the set of vertices has cardinality $k$. For a subset of vertices $S$, this border can be either the vertex border $\delta(S) = N[S] \setminus S$, or the edge border $$\partial(S) = |E(S, \overline{S})| = |\{(u,v) \in E(G): u \in S, v \notin S\}|,$$ which is the set of edges that have exactly one endpoint in $S$. 

We consider the following two standard isoperimetric parameters for graphs:

\[\Phi_E(G,k) = \min_{S\subseteq V : |S|=k} |\partial(S)|, \]
\[\Phi_V(G,k) = \min_{S\subseteq V : |S|=k} |\delta(S)|. \]

The \emph{isoperimetric problem} for either of these two parameters asks for an exact evaluation of $\Phi_E(G,k)$ or $\Phi_V(G,k)$, while an \emph{isoperimetric inequality} is a bound on these values. 
The edge isoperimetric problem is also studied along with the problem of maximizing the number of edges between vertices of a $k$-set of vertices, $\max_{S:|S|=k}|E(S, S)|$. For a survey on edge isoperimetric inequalities, see \cite{edge-iso-survey}.

The two isoperimetric problems are closely related. 
We have that $\Phi_V(G,k) \leq \Phi_E(G,k)$. 
Since each vertex in $\delta(S)$ is incident to at most $\Delta(G)$ vertices in $S$, it follows that $\Phi_E(G,k) \leq \Delta(G) \Phi_V(G,k)$. 
Therefore, these parameters differ at most by a factor of $\Delta(G)$:
\[ \frac{\Phi_E(G,k)}{\Delta(G)} \leq \Phi_V(G,k) \leq \Phi_E(G,k) .
\]

The \emph{isoperimetric peak} is the maximum over the isoperimetric numbers on the graph:
\[\Phi_E(G) = \max_k \Phi_E(G,k),\]
\[\Phi_V(G) = \max_k \Phi_V(G,k).\]
The vertex isoperimetric peak has been explicitly studied for trees \cite{iso-peak-trees-BC-2009,iso-peak-trees-OY-2008,iso-peak-trees-Vrto-2010}, although further results for the isoperimetric peak problem appear implicitly within many works. 

We introduce a modification to this concept inspired by the $h$-index metric of citation metrics. For some function $f:\mathbb{Z}\rightarrow \mathbb{Z}$, define the \emph{$h$-index function} on $f$, $\hindex(f)$, as follows:
\[
\hindex(f) = \max\{h \in \mathbb{Z} : \text{for some }  k_1, \text{ we have } f(k) \geq h\text{ for } k_1 \le k \le k_1+h-1\}.
\]
In particular, there are $\hindex(f)$ consecutive integers $k_1 \le k \le k_1+h-1$ with $f(k)\geq \hindex(f)$. 
The \emph{vertex-$h$-index} of a graph is 
\[\hindex_V(G) = \hindex(\Phi_V(G,k)),\]
and similarly, the \emph{edge-$h$-index} of a graph is 
\[\hindex_E(G) = \hindex(\Phi_E(G,k)).\]
See Figure \ref{fig:h_index} for an illustration of $\hindex_V(G)$.  

The following lemma establishes inequalities for the $\hindex_V$ and $\hindex_E$ parameters.
\begin{lemma} \label{lem:h_index_self_bound}
For a graph $G$, we have that 
\[
\frac{\hindex_E(G)}{\Delta(G)} \leq \hindex_V(G) \leq \hindex_E(G).
\]
\end{lemma}
\begin{proof}
Let $h=\hindex_V(G)$. 
By the definition of $\hindex_V(G)$, there exists an integer $k_a$ such that each $k \in [k_a, \ldots, k_a+h-1]$ satisfies $\Phi_V(G,k) \geq h$. However, since $\Phi_E(G,k) \geq \Phi_V(G,k)$, this gives that $\Phi_E(G,k) \geq h$ for $k \in [k_a, \ldots, k_a+h-1]$. Therefore, $\hindex_E(G)\geq h = \hindex_V(G)$ by the definition of $\hindex_E$. 

Let $h=\hindex_E(G)$. 
By the definition of $\hindex_E(G)$, there exists an integer $k_a$ such that each $k \in [k_a, \ldots, k_a+h-1]$ has $\Phi_E(G,k) \geq h$. However, since $\Phi_V(G,k) \geq \Phi_E(G,k)/\Delta(G)$, this gives that $\Phi_V(G,k) \geq h /\Delta(G)$ for $k \in [k_a, \ldots, k_a+\lceil h/\Delta(G) \rceil-1] \subseteq [k_a, \ldots, k_a+h-1]$. Therefore, $\hindex_V(G)\geq h/\Delta(G)  = \hindex_E(G)/\Delta(G)$ by the definition of $\hindex_V$. 
\end{proof}

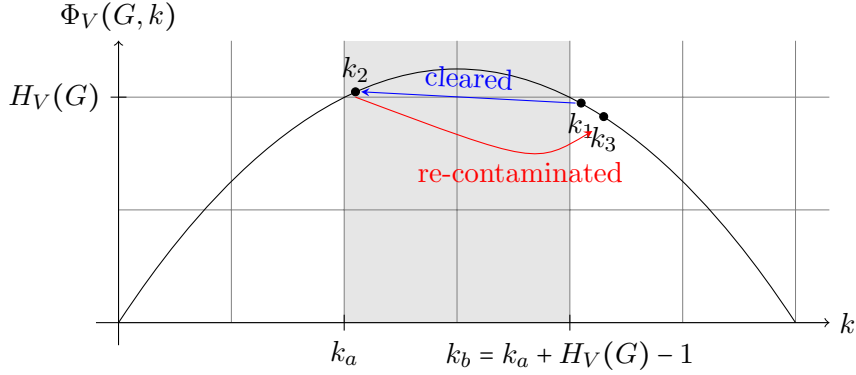
\begin{figure}
    \centering
\begin{tikzpicture}[scale=1.5]
\draw[fill=gray!20] (2,2) parabola bend (3, 2.25) (4,2);
\fill[fill=gray!20] (2,2.5) rectangle (4,0);
  \draw[style=help lines] 
  (0,0) grid (6.3,2.5);
  \draw[->] (-0.2,0) -- (6.3,0) node[right] {$k$};
  \draw[->] (0,-0.2) -- (0,2.5) node[above] {$\Phi_V(G,k)$};

  \draw[shift={(2,0)}] (0pt,2pt) -- (0pt,-2pt) node[below] {$k_a$};
  \draw[shift={(4,0)}] (0pt,2pt) -- (0pt,-2pt) node[below] {$k_b = k_a+\hindex_V(G)-1$};

\draw[shift={(0,2)}] (2pt,0pt) -- (-2pt,0pt) node[left] {$\hindex_V(G)$};

  \draw (0,0) parabola bend (3, 2.25) (6,0);

\filldraw[black] (4.1,1.9475) circle (1pt) node[anchor=north]{$k_1$};
    \filldraw[black] (2.1,2.0475) circle (1pt) node[anchor=south]{$k_2$};
    \filldraw[black] (4.3,1.8275) circle (1pt) node[anchor=north]{$k_3$};

\draw[blue,-{stealth[blue]}] (4.05,1.9475) -- node[above,sloped] {cleared}(2.15,2.0475);
\draw[red,-{stealth[red]}] (2.1,2)  .. controls (3.7, 1.4) .. node[below] {re-contaminated} (4.2,1.7);
    
\end{tikzpicture}
    \caption{A graph of $\Phi(G,k)$ that illustrates $\hindex_V(G)$, where a contiguous set of $\hindex_V(G)$ integers each have $\Phi(G,k) \geq \hindex_V(G)$.  For Theorem~\ref{thm:h_index}, if the cops manage to reduce the number of contaminated vertices $k$ below $k_b+1$ (in our example, moving from $k=k_1$ to $k=k_2$ that is in the gray region), then the contamination is guaranteed to grow to some cardinality at least $k_b+1$ (that is, moving from $k=k_2$ to $k=k_3$ in our example, which is to the right of the gray region).}
    \label{fig:h_index}
\end{figure}

The following theorem gives a lower bound on $\text{prox}_1(G)$ in terms of $\hindex_V(G)$, and hence, gives a lower bound for $\zeta_1(G).$

\begin{theorem} \label{thm:h_index}
    If $G$ is a graph, then
    \[
    \text{prox}_1(G) > \frac{\hindex_V(G)}{\Delta(G)+1}.
    \]
\end{theorem}
\begin{proof}
    We may assume $\hindex_V(G)\geq (\Delta(G)+1)$, or else the proof is immediate.  
    See Figure \ref{fig:h_index} as an aid to this proof. 
    As a high-level overview of the proof, $p$ cops can clear at most $p (\Delta(G)+1)$ vertices per round (that is, a cop on $u$ clears the at most $\Delta+1$ vertices in $N[u]$)  and the contamination spreads at a rate of at least
    $\Phi_V(G,k)$, with $k$ being the number of currently contaminated vertices. 
    If $\Phi_V(G,k)$ is larger than $p  (\Delta(G)+1)$ for enough contiguous values $k$, then there will always be some point in the game where the contamination will grow faster than the cops can clear the contamination. 

    Consider a game in which the cop player controls $p=\big\lfloor \frac{\hindex_V(G)}{\Delta(G)+1} \big\rfloor$ cops. 
    By the definition of $\hindex_V(G)$, there must exist $k_a$ and $k_b=k_a+\hindex_V(G)-1$ such that for any value $k\in \{k_a, k_a+1, \ldots, k_b\}$ we have $\Phi_V(G,k)\geq \hindex_V(G)$. 
    Note that $k_a>1$, since $$\Phi_V(G,1)=\min_{v\in V(G)} \deg(v)  <  \Delta(G)+1 \leq \hindex_V(G),$$ so the inequality $\Phi_V(G,1)\geq \hindex_V(G)$ does not hold. 

    Suppose there are at least $k_b+1$ contaminated vertices just before the cops move.  
    If a cop plays on vertex $v$, then the vertices on $N[v]$ that were contaminated are no longer contaminated.  
    As a consequence, after this cop round
    at most $p(\Delta(G)+1) \leq \hindex_V(G)$ vertices have been cleared, which implies that at least $$k_b +1-\hindex_V(G) = (k_a+\hindex_V(G)-1)+1-\hindex_V(G) = k_a$$ vertices remain contaminated.
    That is, in a round where the cops reduce the contaminated vertices below $k_b+1$, there will always be at least $k_a$ contaminated vertices remaining. 
    For the cops to win by eliminating all contaminated vertices, there must be some round where between $k_a$ and $k_b$ vertices are contaminated. Suppose we are in such a round after the cops move. 
    
    Now $\Phi_V(G,k) \geq \hindex_V(G)$ for each $k$ with $k_a \leq k \leq k_b$, so on the contamination round at least $\hindex_V(G)$ clear vertices become re-contaminated, and so at least $k_a+\hindex_V(G) = k_b+1$ vertices are now contaminated. 
    That is, if the cops ever reduce the number of contaminated vertices to be $k_b$ or fewer, then there will be $k_b+1$ or more contaminated vertices after the subsequent contamination round. 
    Therefore, we conclude that the cops can never reduce the number of contaminated vertices below $k_a$.
\end{proof}

An analogous bound for $\hindex_E$ can be found using Lemma~\ref{lem:h_index_self_bound} on Theorem~\ref{thm:h_index} is the following.

\begin{corollary} \label{cor:h_index_edges}
        For a graph $G$, 
    \[
    \prox_1(G) > \frac{\hindex_E(G)}{(\Delta(G)+1)\Delta(G)}.
    \]
\end{corollary}

Since $\zeta_1(G) \geq \prox_1(G)$, both Theorem~\ref{thm:h_index} and Corollary~\ref{cor:h_index_edges} allow results in isoperimetric bounds to be applied to yield lower bounds on the $k$-visibility Location number. The primary challenge remaining is that even when the isoperimetric parameters are known exactly, computing $\hindex_V(G)$ and $\hindex_E(G)$ is often a complex task. As an example, the vertex isometric peak of a binary tree of radius $d$ is known to be asymptotically equal to $d/2$ \cite{HY2015}; it is also known that the number of vertices in the vertex border is small for many sporadic values of $k$. 
However, we can find a lower bound of $H_V$ and $H_E$ using the corresponding isoperimetric peak, showing that an $h$-index and its corresponding isoperimetric parameter differ by a small multiplicative constant without complicated direct analysis.

\begin{theorem} \label{thm:peak_to_h_V}
    For a graph $G$, 
    \[
         \frac{\Phi_V(G)}{2} \left(1 + \frac{1}{2\Delta(G) + 1}\right) \leq \hindex_V(G)  \leq \Phi_V(G).
    \]
\end{theorem}
\begin{proof}
The upper bound is clear by the definition of $H_V$ and $\Phi_V$. For the lower bound, we will show that two consecutive values cannot have their $\Phi_V(G,k)$ values differ too much.    

Consider a set $S$ of cardinality $k+1$. If we remove any vertex $u$ from $S$, then the only vertex that can be in $\delta(S\setminus\{u\})$ that is not in $\delta(S)$ is the vertex $u$, and so 
$|\delta(S)| \geq |\delta(S\setminus\{u\})|-1$. 
As a consequence, if $|\delta(S)| = \Phi_V(G,k+1)$ and $u \in S$, then we have that
\begin{align} \label{align_phi_v}
\Phi_V(G,k+1) = |\delta(S)| \geq |\delta(S\setminus\{u\})| -1 \geq \min_{S':|S'|=k} |\delta(S')|-1 = \Phi_V(G,k)-1. 
    \end{align}

Now consider a set $S$ of cardinality $k-1$. If we add a vertex $v$ to $S$, then any vertex in $\delta(S \cup \{v\})$ that is not in $\delta(S)$ must be a neighbor of $v$. As a consequence, $\delta(S \cup \{v\})$ contains at most $\Delta(G)$ more vertices than $\delta(S)$, and so similar to the previous case, \begin{align} \label{align_phi_e}
\Phi_V(G,k-1) \geq \Phi_V(G,k)-\Delta(G) .
\end{align}

Let $k_p$ be a value such that $\Phi_V(G,k_p)=\Phi_V(G)$. 
By recursively applying inequality \eqref{align_phi_v}, we find htat
$$\Phi_V(G,k_p+i) \geq \Phi_V(G,k_p) - i \geq \Phi_V(G) - \Phi_V(G) \Delta(G)/ (2\Delta(G)+1)$$ for each $i \in \{0, \ldots, \Phi_V(G) \Delta(G)/ (2\Delta(G)+1)\}$.
Similarly, by recursively applying inequality \eqref{align_phi_e},
\[\Phi_V(G,k_p-i) \geq \Phi_V(G,k_p) - i\Delta(G)\geq \Phi_V(G) - \Phi_V(G)\Delta(G)/(2\Delta(G)+1)\]
for each $i \in \{1, \ldots, \Phi_V(G)/ (2\Delta(G)+1)\}$.
Therefore, there are at least
\[
\frac{\Phi_V(G) \Delta(G)}{ 2\Delta(G)+1}+1 + \frac{\Phi_V(G)}{ 2\Delta(G)+1} = \Phi_V(G)\frac{ \Delta(G) +1}{ 2\Delta(G)+1}+1
\]
contiguous values of $k$ such that 
\[\Phi_V(G,k) \geq \Phi_V(G) - \frac{\Phi_V(G)\Delta(G)}{2\Delta(G)+1} = \Phi_V(G) \frac{\Delta(G)+1}{2\Delta(G)+1}.\]
By the definition of $\hindex_V$, this yields that 
\[\hindex_V(G) \geq \Phi_V(G) \frac{\Delta(G)+1}{2\Delta(G)+1} = \frac{\Phi_V(G)}{2} \left(1 + \frac{1}{2\Delta(G)+1}\right),\]
as required. 
\end{proof} 

We note that as $\Phi_V(G) \geq \hindex_V(G)$, the vertex-$h$-index and the vertex isoperimetric peak differ by a multiplicative factor of at most a little over $2$. There is an analogous result for the edge-$h$-index, as follows. 
\begin{theorem} \label{thm:peak_to_h_E}
    For a graph $G$, 
    \[
        \hindex_E(G) \geq \frac{2}{\Delta(G) + 2} \Phi_E(G) .
    \]
\end{theorem}
\begin{proof}
    Follows similarly to the proof of Theorem~\ref{thm:peak_to_h_V}, except
    using
    \[ \Phi_E(G,k+1) \ge \Phi_E(G,k)-\Delta(G) \]
    in place of the inequality \eqref{align_phi_v}. 
\end{proof}
We have the following corollary as a consequence of Theorems~\ref{thm:h_index} and \ref{thm:peak_to_h_V}.
 
\begin{corollary}
For a graph $G$, we have that 
    \[
    \prox_1(G) = \Omega\left(\frac{\Phi_V(G)}{\Delta(G)}\right),
    \]
    and as $\Phi_V(G) \ge \Phi_E(G)/\Delta(G)$,
        \[
    \prox_1(G) = \Omega\left(\frac{\Phi_E(G)}{\Delta(G)^2}\right).
    \]  
\end{corollary}

\section{Trees}\label{sec:trees}

As was proved first in \cite{seager1} and later in \cite{nisse2}, the localization number of trees is at most two. For the one-visibility Localization game, the situation is quite different. We explore bounds on $\zeta_1$ for trees and give upper and lower bounds on $\zeta_1$ for $k$-ary trees with $k\ge 2$ that differ by a multiplicative constant; this family is shown as a result to have $\zeta_1$ unbounded. 

We begin by showing that $\zeta_1$ is monotone on subtrees of a tree.

\begin{lemma}
If $T$ is a tree and $S$ is a subtree of $T$, then $\zeta_1(S) \le \zeta_1(T)$.
\end{lemma}

\begin{proof}
For each $v \in T$, let $\textrm{des}(v) \in S$ be the unique vertex in $S$ at the shortest distance (in $T$) from $v$. Using $\zeta_1(T)$ cops, if a successful strategy on $T$ calls for a cop to probe $v \in T$, they instead probe $\textrm{des}(v)$. Since the robber cannot be in $T \setminus S$, the distance in $T$ between the robber and $\textrm{des}(v)$ is at most the distance between the robber and $v$. Thus, this strategy gives no less information than it would in $T$. As the cops win in $T$, they will win in $S$. 
\end{proof}

Our next result shows that the one-visibility localization number and one-proximity number differ by at most one for trees.

\begin{lemma}\label{lem:prox_bound}
For any tree $T$,
\[ \prox_1(T) \le \zeta_1(T) \le \prox_1(T)+1. \]
Furthermore, if $\prox_1(T) \ge \Delta(T)$, then $\zeta_1(T) = \prox_1(T)$.
\end{lemma}

\begin{proof}
Let $m = \prox_1(T)$. Root $T$ arbitrarily and let $r$ be the root vertex. The cops of the one-visibility Localization game use $m$ cops to follow a winning strategy in the one-proximity game and place one additional cop at $r$ on each round. 
If the robber ever tries to cross from one subtree of $T- r$ to another, then they must pass through $r$, at which point the cop on $r$ would probe a distance of $0$, and the robber would be captured. Therefore, the robber may only move in one subtree. 

As the cops' strategy succeeds in the one-proximity game, they eventually probe a vertex and receive a distance of at most one. If the probe returns distance zero, the robber has been located, so assume it returns one. This uniquely determines on which subtree the robber is located. 
If the cop on $r$ also returned a probe of distance $1$ to the robber, then the robber would be captured as $r$ has exactly one neighbor on each subtree. Therefore, we may assume the robber has distance at least $2$ from $r$.  
The cop dedicated to probing $r$ can now probe the root of the robber's subtree while the other $m$ cops start the winning strategy from the beginning. After repeating this process at most $d$ times, where $d$ is the depth of $T$, the subtree onto which the robber is forced will be a leaf. The cops can then win by probing each leaf until locating the robber.
\end{proof}

Let $T$ be a tree. Note that given any $v \in T$, $T - v$ is a forest. Call $v$ a \emph{midway vertex} of $T$ if each component $T_1, T_2, \ldots, T_k$ of $T - v$ satisfies $|V(T_i)| \le n/2$. The following result is folklore.
\begin{lemma}
Every tree has at least one midway vertex.
\end{lemma}

\begin{proof}
Given $u \in V(T)$, let $T_1, T_2, \ldots, T_k$ be the components of $T-u$ and define $$s(u) = \max_{1\le i \le k} |V(T_i)|.$$ 
Assume for the sake of contradiction that $\min_{u \in V(T)} s(u) > n/2$. 
Let $v$ be a vertex with $s(v)$ minimal and let $T_1, T_2, \ldots, T_k$ be the components of $T - v$, where we have $s(v)>n/2$ by our initial assumption. Note that at most one component satisfies $|V(T_i)| > n/2$ as $\sum |V(T_i)| = n-1$; without loss of generality, let this component be $T_1$. 

Let $w \in T_1$ be the neighbor of $v$ in $T$. We then have that $T-w$ is a collection of components, say $S_1, S_2,\ldots, S_r$. The largest of these components, say $S_1$, cannot be a subtree of $T_1$ as any such subset does not contain $w$ and therefore has fewer than $|V(T_1)|$ vertices, contradicting that $s(v)$ was minimal. However, any $S_i$ intersecting $T_1$ must, in fact, be a subtree of $T_1$ as the only path from elements in $T_1$ to elements in $T \graphminus T_1$ contains $w$. 

Thus, the largest component $S_1$ must be the subtree formed by combining the subtrees $T_2, \ldots, T_k$ and $v$, and so we have that $|V(S_1)| \le n-|V(T_1)|$. As $|V(T_1)| > n/2$, we have that
\[ |V(S_1)| \le n-|V(T_1)| < n-n/2 = n/2, \]
which contradicts that $s(w) \ge s(v)$.
\end{proof}

We now derive the following bound in terms of the order of the tree.  
\begin{theorem}\label{bn}
If $T$ is a tree of order $n \ge 2$, then $\zeta_1(T) \le \lceil \log_2 n \rceil.$
\end{theorem}

\begin{proof}
The proof is by induction on $n$. The base case with $n=2$ is straightforward: the only tree on two vertices is an edge, for which $\log_2 2 = 1$ probe suffices.

Now assume $T$ is a tree on $n$ vertices. Let $x$ be a midway vertex of $T$ and probe $x$ every round. This prevents the robber from moving from one component of $T-x$ to another. Each component of $T-x$ contains at most $n/2$ vertices and thus, by induction, requires at most $\log_2(n/2) = \log_2 n - 1$ probes to search. As the robber is restricted to a single component of $T-x$, the cops can use these $\log_2 n - 1$ probes to clear each component of $T-x$ before moving on to the next. The winning cop strategy we outlined uses at most $1+\log_2 n - 1 = \log_2 n$ probes, and the proof follows.
\end{proof}

The \emph{depth} of a vertex in a rooted tree is the number of edges in a shortest path from the vertex to the tree's root. The \emph{depth} of a rooted tree $T$ is the greatest depth in $T.$ The \emph{depth} of a tree is the smallest depth of a rooted tree over all ways of rooting $T.$
We next turn to two bounds in terms of the depth of a tree.

\begin{theorem} \label{thm:kAryTree_Upper}
For a tree $T$ of depth $d$, we have that
\[ \zeta_1(T) \le \left\lfloor \frac{d}{4} \right\rfloor + 2. \]
\end{theorem}

\begin{proof}
We provide a strategy using $m=\lfloor \frac{d}{4} \rfloor + 1$ cops to win the one-proximity game on $T$. The result then follows from Lemma~\ref{lem:prox_bound}. The idea of the proof is to clear paths sequentially, based on an ordering of the leaves. We clear all the paths to leaves from lower to higher index, using two cop moves to clear a given path. We ensure that the robber cannot reinfect previously infected paths to leaves with a lower index in each round.

Let $u_0$ be the root vertex of $T$,  and let $\ell_1, \ldots, \ell_p$ denote an ordering of the leaves of $T$, where the ordering is obtained by performing a depth-first search on $T$. Let $i$ be the smallest index such that $\ell_i$ has not yet been chosen.  Let $P_i=u_0u_1, \ldots, u_q=\ell_i$ denote the path from the root to $\ell_i$. Define $v_i$ as the vertex in $P_i$ that is not in $P_{i+1}$ but is as close to the root $u_0$ as possible. 

By induction, we assume that an even number of rounds have occurred, and it is immediately before the cops' move on a round of odd parity. Further, we assume that each subtree in the forest $T - P_i$ either has:  
\begin{enumerate}
\item all vertices cleared; or 
\item all vertices infected, except perhaps the unique vertex with a neighbor in $P_i$ (within the graph $T$).
\end{enumerate}
We call these subtrees \emph{cleared} and \emph{infected}, respectively. The base step for induction is follows as, on the first round, $T - P_1$ is composed of infected trees.  
We note that under these assumptions and since the $\ell_i$ were defined using a depth-first search, directly before the cops take their move, the only infected descendants of $v_i$ are in $P_i$. 

Recall that $q\leq d$ is the length of the path $P_i$. 
The cop player places a cop on the vertex $u_{4j}$ for each $0 \leq j \leq \lfloor q/4\rfloor$ using at most $\lfloor q/4\rfloor + 1 \le m$ cops. If the robber was on $P_i$ and was not captured in this move, then they must have been on $u_{4j+2}$ for some $0 \leq j \leq \lfloor (q-2)/4\rfloor$ or else on $u_q=\ell_i$ if $q \equiv 3 \ (\text{mod } 4)$. Therefore, after the robber moves, they are on either an infected subtree of $T-P_i$ or on $N[u_{4j+2}]$ for some $0 \leq j \leq \lfloor (q-2)/4\rfloor$. On the next cop move, a cop is placed on vertex $u_{4j+2}$ for each $0 \leq j \leq \lfloor (q-2)/4\rfloor$ using at most $\lfloor (q-2)/4\rfloor + 1 \le m$ cops. 
The robber may now only be on an infected subtree and is not on $P_i$ nor on a cleared tree.  
As we noted before these two rounds, the descendants of $v_i$ were infected only if they were in $P_i$. Consequently, each descendant of $v_i$ (including $v_i$ itself) is cleared after these two moves. The robber then takes their move, and after this move may be on an infected tree or on a vertex in $V(P_i) \cap V(P_{i+1})$. 

We observe that the descendants of $v_i$ (including $v_i$ itself) form a cleared tree of $T- P_{i+1}$, which we label as $S$. 
In addition, a cleared tree in $T- P_{i}$ is either a cleared tree in $T - P_{i+1}$ or is a subtree of $S$. 
Similarly, we can show that each subtree of $T - P_{i+1}$ that is not one of these cleared trees is an infected tree. 
Therefore, we have that an even number of rounds has occurred, it is now the cops' move on an odd-parity round, and each subtree of $T - P_{i+1}$ is either cleared or infected, completing the inductive step. 
\end{proof}

We provide a third and final upper bound for $\zeta_1$ on trees. An example will follow, illustrating three graphs such that each bound is best on exactly one graph. 

Let $T$ be a tree, rooted at vertex $v$, of depth $d$. Let $\mathcal{L}_v = \{ L_1, L_2, \dots, L_d\}$ be a \emph{level decomposition} of $T$ rooted at $v$, where $L_i = \{u \in V(T) : d(u,v) = i\}$. Define the nonnegative integer $\overline{L_i} = |\{w \in L_i : \deg(w) \ge 2\}|,$ which counts the number of non-leaf vertices within each level. We have the following upper bound on $\prox_1(T)$.

\begin{theorem}\label{levt}
If $\mathcal{L}_v$ is the level decomposition of $T$ rooted at $v$, then
\[ \prox_1(T) \le \left\lceil \frac{\max_{i} \{\overline{L_i}\}}{3} \right\rceil + 1 .\]
\end{theorem}

\begin{proof}
Let $k=\left\lceil \frac{\max_{i}\{\overline{L_i}\}}{3} \right\rceil$. We give a strategy for $k$ + 1 one-proximity cops to clear $T$ starting from level $L_{m-1}$ and working up to the root $v$. Divide the non-leaf vertices of $L_{m-1}$ into disjoint groups of three vertices, assigning one cop to each set, and consider the first such group consisting of vertices $x$, $y$, and $z$. For $x$, let $x_1$ denote the parent of $x$, $x_2$ denote the parent of $x_1$, and define $y_1, y_2, z_1$ and $z_2$ analogously. We refer to the cop initially assigned to these three vertices as $C_1$. At the start of the game, the robber may be anywhere on $T$, so all vertices start contaminated. 

In the first three rounds, $C_1$ will probe $x$, then $y$, then $z$. After these probes, immediately before the robber's move, the robber may be at $x$ (had they begun on $x_2$, they could move to $x_1$ prior to the cop probing $y$, and then to $x$ prior to the cop probing $z$). The robber could also be at $y_1$ or $z_2$. We note that the robber cannot currently be on a leaf adjacent to $x$, $y$, or $z$ without being previously detected by one of the first three probes. See Figure~\ref{L_m-1}.

\begin{figure}[ht!]
\centering
\begin{tikzpicture}
\node[fill=red, circle, label=west:$x$] (x) at (1,1) {};
\node[draw, fill=white, circle] (x1) at (0,0) {};
\node[draw, fill=white, circle] (x2) at (0.5,0) {};
\node at (1.25,0) {$\cdots$};
\node[draw, fill=white, circle] (x3) at (2,0) {};
\node[fill=red, circle, label=west:$x_1$] (Px) at (1,2) {};
\node[fill=red, circle, label=west:$x_2$] (PPx) at (1,3) {};
\draw (PPx) -- (Px) -- (x) -- (x1);
\draw (x) -- (x2);
\draw (x) -- (x3);

\node[draw, fill=white, circle, label=west:$y$] (y) at (4,1) {};
\node[draw, fill=white, circle] (y1) at (3,0) {};
\node[draw, fill=white, circle] (y2) at (3.5,0) {};
\node at (4.25,0) {$\cdots$};
\node[draw, fill=white, circle] (y3) at (5,0) {};
\node[fill=red, circle, label=west:$y_1$] (Py) at (4,2) {};
\node[fill=red, circle, label=west:$y_2$] (PPy) at (4,3) {};
\draw (PPy) -- (Py) -- (y) -- (y1);
\draw (y) -- (y2);
\draw (y) -- (y3);

\node[draw, fill=white, circle, label=west:$z$] (z) at (7,1) {};
\node[draw, fill=white, circle] (z1) at (6,0) {};
\node[draw, fill=white, circle] (z2) at (6.5,0) {};
\node at (7.25,0) {$\cdots$};
\node[draw, fill=white, circle] (z3) at (8,0) {};
\node[draw, fill=white, circle, label=west:$z_1$] (Pz) at (7,2) {};
\node[fill=red, circle, label=west:$z_2$] (PPz) at (7,3) {};
\draw (PPz) -- (Pz) -- (z) -- (z1);
\draw (z) -- (z2);
\draw (z) -- (z3);

\node at (-1,0) {$L_m$};
\node at (-1,1) {$L_{m-1}$};
\node at (-1,2) {$L_{m-2}$};
\node at (-1,3) {$L_{m-3}$};
\end{tikzpicture}
\caption{Possible robber locations are shown in red after $C_1$'s third probe.}\label{L_m-1}
\end{figure}
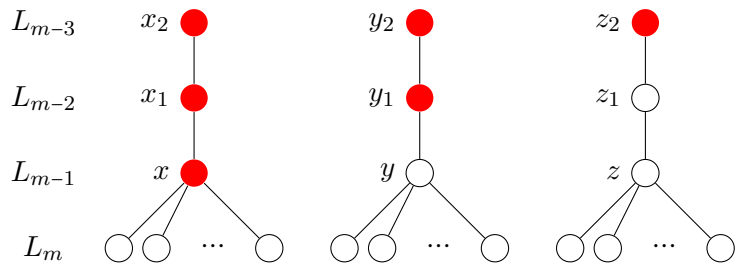

The robber now takes their move. To keep the leaves adjacent to $x$ clear, $C_1$ will next probe $x$. They could then probe $y$, and then $z$, repeating the previous three probes. This would protect the leaves adjacent to $x$, $y$, and $z$, but would not allow the cop player to make progress. We now introduce our additional cop, $C^*$, who will help $C_1$ shift from guarding $x,y$ and $z$ to guarding $x_1, y_1$, and $z_1$. To do this, $C^*$ probes $x$ instead of $C_1$. Had the robber moved to $x$ they will be detected by this next probe regardless of where they now move. This allows $C_1$'s next three probes to be $x_1$, then $y$, then $z$, while $C^*$ continues to probe $x$. Following these three rounds of cop probes, $x$ is fully cleared. 

In the next three rounds we will have $C^*$ probe $y$, while $C_1$ probes $x_1$, then $y_1$, then $z$. Next, $C^*$ probes $z$ while $C_1$ probes $x_1$, then $y_1$, then $z_1$. On the robber's move following this sequence of probes, they may now be at $x_1$, $y_2$, or $z_3$. Next, $C^*$ will move to the second trio of non-leaf vertices in $L_{m-1}$ and will perform this same strategy with the cop there, $C_2$. At the same time, $C_1$ will repeatedly probe $x_1$, then $y_1$, then $z_1$, ensuring that the robber can only reach $L_{m-1}$ in this part of the graph, and will be detected doing so, thereby clearing $x$, $y$, and $z$, and the adjacent leaves in $L_m$. This is continued for each group of three non-leaf vertices in $L_{m-1}$, of which we have at most $k$. 

When $C^*$ has concluded with the last group of three non-leaves in $L_{m-1}$, the cops will probe in $L_{m-1}$ so that every non-leaf vertex there is probed every three rounds. This ensures that the robber will be captured if they ever move onto a non-leaf in $L_{m-1}$.

The clearing strategy continues up the tree, one level at a time. When the cops move to $L_{m-2}$, we note that we may have non-leaf vertices in $L_{m-2}$ which have yet to be probed (those that have all their children as leaves in $L_{m-1}$). The tree $T$ has at most $k$ sets of three non-leaves in any level. For such vertices, divide them into groups of three and assign an unused cop to each. We have these cops repeatedly probe each of their three vertices (as $C_1$ probed $x$, $y$, then $z$ initially), clearing their adjacent leaves in $L_{m-1}$.

The cop $C^*$ is then used to extend the cop territory up from $L_{m-2}$ into $L_{m-3}$. First, $C^*$ probes $x_1$ while $C_1$ probes $x_2$, then $y_1$, then $z_1$. The process is repeated as in the level below. In each grouping of vertices, the cop $C^*$ extends the cop territory up the tree, one vertex at a time. Since each vertex in a level is probed every three terms while $C^*$ is probing that level, the robber cannot move down the tree into cop territory without being detected.

In this way, $C^*$ can extend the cop territory eventually up to the root $v$, cleaning the tree and winning the one-proximity game. The proof follows.
\end{proof}

For integers $d,k \ge 2,$ we use the notation $T^k_d$ for the $k$-ary tree of depth $d,$ where each non-leaf vertex has $k$ children. We note that the bound in Theorem~\ref{levt} is tight for $T^3_2$. 
Rooted at the midway vertex, the tree has $\max_i \{\overline{L_i} \}=|L_1| = 3$, and so Theorem~\ref{levt} provides that $\prox_1(T^3_2) \leq 2$. It is known~\cite{seager2} that the original localization number satisfies $\zeta(T^3_2)=2$, so we have $\zeta_1(T^3_2)\geq 2$ as well.

 To compare the upper bounds on $\zeta_1(T)$ provided by Theorems~\ref{bn}, \ref{thm:kAryTree_Upper}, and \ref{levt}, let $T_0=T^3_3$, which has 40 vertices and 39 edges. The corresponding upper bounds for $T_0$ are displayed in the first row of Table~\ref{tab1_new}. Define the trees $T_i$ to be $T_0$ with each edge subdivided into $i$ vertices. We then have that $T_i$ will have depth $3(i+1)$ and order $40+39i$. The best upper bound for each tree is shown in bold. Note that each of the theorems gives the best bound depending on the tree considered.
\begin{center}
\begin{table}[h]
\begin{tabular}{|c|c|c|c|}
\hline
          & Theorem~\ref{bn}             & Theorem~\ref{thm:kAryTree_Upper}   & Theorem~\ref{levt} \\ \hline
$T_0$     & $6$           & \textbf{2}      & 4         \\ \hline
$T_{10}$     & \textbf{9}    & 10           & 10         \\ \hline
$T_{100}$ & $12$        & 77          & \textbf{10}        \\ \hline
\end{tabular}
\caption{Upper bounds for $\zeta_1$ from Theorems~\ref{bn}, \ref{thm:kAryTree_Upper}, and \ref{levt}.}\label{tab1_new}
\end{table}
\end{center} 

Finding lower bounds for the one-visibility localization number is challenging in most cases.  We determined that the isoperimetric peak can give a lower bound on $\prox_1$, enabling us to utilize such isoperimetric results when they exist. We finish by applying such results to binary trees, where $k=2$. We cite the following result, which gives asymptotically tight values for the isoperimetric peak of binary trees. 

\begin{theorem}[\cite{HY2015}]\label{tbdd}
If $d\ge 2$ is an integer, then
\[ \frac{d}{2} -O(\log{d}) \le  \Phi_E(T^2_d) \le \frac{d}{2} +O(1) .
\]
\end{theorem}

We therefore have that $\hindex_E(T^2_d)\geq \frac{d}{5}-O(\log{d})$ from Theorem~\ref{thm:peak_to_h_E}, and so we have the following bounds. 
\begin{corollary}\label{tboundd}
If $d\ge 2$ is an integer, then

\[
\frac{d}{60} - O(\log{d}) < \prox_1(T^2_d)  \le \zeta_1(T^2_d) \le  \frac{d}{4}+2 .
\]
\end{corollary}
\begin{proof} The upper bound follows from Theorem~\ref{thm:kAryTree_Upper}. The lower bound follows from using Corollary~\ref{cor:h_index_edges}, then applying Theorem~\ref{thm:peak_to_h_E} to the result, and then finally using the lower bound of Theorem~\ref{tbdd}. \end{proof}

 Similar lower bounds of the vertex isoperimetric peak on $k$-ary trees are also useful to us here. 

\begin{theorem}[\cite{iso-peak-trees-Vrto-2010}]\label{thm:k_tree_vertex_isop_lower}
If $d,k\ge 2$ are integers, then 
\[ \Phi_V(T^k_d) \geq \frac{3}{40} (d-2).
\]
\end{theorem}

Theorem~\ref{thm:k_tree_vertex_isop_lower} provides the following bounds on $\prox_1$ and $\zeta_1$ for $k$-ary trees.
\begin{corollary} \label{cor:zeta_kAry_upper}
If $d,k\ge 2$ are integers, then
\[
\frac{3}{80} (d-2) \left( \frac{2}{2k+3} \right) < 
 \prox_1(T^k_d) \leq \zeta_1(T^k_d) \leq \frac{d}{4}+2.
\]
\end{corollary}
\begin{proof}
The upper bound follows from Theorem~\ref{thm:kAryTree_Upper}.
The lower bound follows from using Theorem~\ref{thm:h_index}, then applying Theorem~\ref{thm:peak_to_h_V}, and then finally using Theorem~\ref{thm:k_tree_vertex_isop_lower}. 
\end{proof}

Corollaries~\ref{tboundd} and ~\ref{cor:zeta_kAry_upper} provide families of trees with unbounded $\zeta_1$ number, in stark contrast to $\zeta$ being bounded by 2 for trees. The result of Corollary~\ref{cor:zeta_kAry_upper} is far from tight. We note that improvements to the isoperimetric value of $k$-ary trees would improve this result. 

\section{Cartesian grid graphs} \label{sec:grids}

We proved in Theorem~\ref{plan} that for planar graphs of order $n,$ $\zeta_1(G) \le O(\sqrt{n})$. In this section, we show that grid graphs make this bound tight, in the sense that such graphs have $\zeta_1$ numbers in $\sqrt{|V(G)|}+O(1)$. 

For a positive integer $n$, let $G_{n,n}$ be the $n \times n$ \emph{Cartesian grid}, which consists of the Cartesian product of the $n$-order path with itself, or $P_n \square P_n.$ As we only consider Cartesian grids, we refer to them as \emph{grids}. 

The lower bound for grids follows by using our earlier results with the $h$-index. 

\begin{theorem}\label{thm:grid_lower}
For a positive integer $n>1$, $\prox_1(G_{n,n}) \geq \frac{n}{5}+1$.
\end{theorem}
\begin{proof}
The vertex isoperimetric values are known for the grids \cite{grids-BL-1991}. In particular, for $G_{n,n}$, $\Phi_V(G_{n,n},k)=n$ for $k\in \{\frac{n^2 - 3n + 4}{2}, \ldots, \frac{n^2+n-2}{2} \}$, which are $2n-2$ contiguous values of $k$. 
Thus, it follows that $\hindex_V(G_{n,n}) \geq n$, and by Theorem~\ref{thm:h_index} we have that $\prox_1(G_{n,n}) > n/5$, as required. 
\end{proof}

We next establish upper bounds for grid graphs that differ from the lower bound in Theorem~\ref{thm:grid_lower} by an additive constant.

\begin{theorem} \label{thm:grid_upper}
Let $m$ be the odd integer such that $n=5m-i$ for some integer $0\le i \le 9$. 
We then have that
\[\prox_1(G_{n,n}) \leq m+3.\]
\end{theorem}

We prove Theorem~\ref{thm:grid_upper} at the end of this section. As a consequence of Theorem~\ref{thm:grid_upper}, we know $\prox_1(G_{n,n})$ up to one of four values. This gives a class of graphs where Theorem~\ref{thm:h_index} is close to being tight. 

\begin{corollary} \label{cor:prox_grids_both}
For $n$ a positive integer, 
\[ 
\Big\lceil\frac{n}{5}
\Big\rceil+1 \leq \prox_1(G_{n,n}) 
\leq 
\Big\lceil\frac{n}{5}
\Big\rceil+4.\]
\end{corollary}
\begin{proof}
Theorem \ref{thm:grid_lower} gives $\prox_1(G_{n,n})\geq \lceil\frac{n}{5} \rceil+1$. 
When we note that $m=\lceil \frac{n}{5}\rceil + \lfloor \frac{i}{5}\rfloor \leq \lceil \frac{n}{5}\rceil+1$ in Theorem \ref{thm:grid_upper}, it follows that $\prox_1(G_{n,n}) \leq \lceil \frac{n}{5}\rceil+4$. 
\end{proof}

The following result determines the one-visibility localization number of Cartesian grids up to four values, assuming that $n$ is sufficiently large. 

\begin{corollary}
If $n\geq 11$, then 
\[ 
\Big\lceil\frac{n}{5}
\Big\rceil+1 \leq \zeta_1(G_{n,n}) 
\leq 
\Big\lceil\frac{n}{5}
\Big\rceil+4.\]
\end{corollary}
\begin{proof}
Using $\prox_1(G_{n,n})$ cops in the one-visibility Localization game, the cops may initially follow the strategy guaranteed by Corollary~\ref{cor:prox_grids_both} until some cop probes a distance of $1$ from the robber, say the cop on vertex $u$. The robber moves and may now be on any vertex of distance $0$, $1$, or $2$ from $u$. Since $n\geq 11$, we have that $\lceil \frac{n}{5} \rceil +1 \geq 4$, and so we know that at least four cops are playing. For the cops' second move, play four cops on the vertices of $N(u)$. 

If all four cops probe a distance of $1$, then the robber is on $u$. If exactly two of these cops probe a distance of $1$, then the robber must be on the unique vertex of distance $2$ from $u$ that is
adjacent to the two vertices containing these cops. If exactly one cop probes a distance of $1$, then the robber must be on the unique vertex of distance $1$ from this cop that is not adjacent to any of the other cops. Thus, the robber's location is determined by the cops.   
\end{proof}

We finish the section with the proof of Theorem~\ref{thm:grid_upper}, but as the proof is quite technical, we motivate our technique by exploring some less efficient but more intuitive strategies. A first strategy one may think of for the cops is to clear the grid using $n$ cops to probe an entire row and march upwards, clearing the grid from bottom to top. This approach is effective but inefficient.

A natural improvement is to place a cop in every other column, alternating between the first and second row, so that two full rows of the grid are cleared using at most $(n+1)/2$ cops; see Figure~\ref{fig:interlock}.
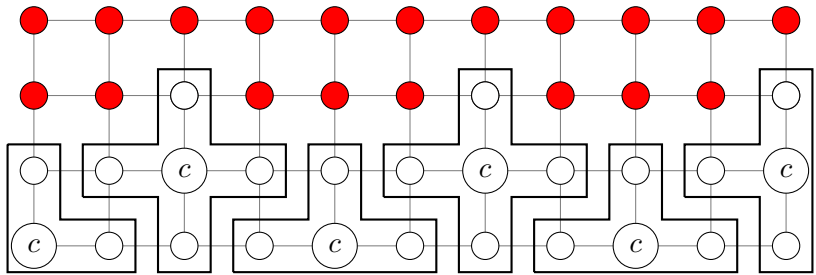
\begin{figure}[htpb!]
    \centering
    \begin{tikzpicture}
    \draw[thin,gray] (0,1) grid (10,4);

    \foreach \x in {0,...,10} {
        \foreach \y in {4,3} {
            \node[draw,circle,fill=red] at (\x,\y) {};
        }    
    }

    \node[draw,circle,fill=white] at (0,1) {$c$};
    \node[circle,draw,fill=white] at (1,1) {};
    \node[circle,draw,fill=white] at (0,2) {};
    \draw[thick] (0-0.35,2.35) -- 
    (0+0.35,2.35) -- (0+0.35,1.35) -- (0+1.35,1.35) -- 
    (0+1.35,0.65) -- (0-0.35,0.65) -- (0-0.35,2.35);
        
    \foreach \x in {4,8} {
        \node[draw,circle,fill=white] at (\x,1) {$c$};
        \node[circle,draw,fill=white] at (\x+1,1) {};
        \node[circle,draw,fill=white] at (\x,2) {};
        \node[circle,draw,fill=white] at (\x-1,1) {};
        \draw[thick] (\x-1.35,0.65) -- (\x-1.35,1.35) --
        (\x-0.35,1.35) -- (\x-0.35,2.35) -- 
        (\x+0.35,2.35) -- (\x+0.35,1.35) -- 
        (\x+1.35,1.35) -- (\x+1.35,0.65) -- 
        (\x-0.35,0.65) -- (\x-1.35,0.65);
    }

    \foreach \x in {2,6} {
        \node[draw,circle,fill=white] at (\x,2) {$c$};
        \node[circle,draw,fill=white] at (\x-1,2) {};
        \node[circle,draw,fill=white] at (\x,3) {};
        \node[circle,draw,fill=white] at (\x,1) {};
        \node[circle,draw,fill=white] at (\x+1,2) {};
        \draw[thick] (\x-1.35,1.65) -- (\x-1.35,2.35) --
        (\x-0.35,2.35) -- (\x-0.35,3.35) -- 
        (\x+0.35,3.35) -- (\x+0.35,2.35) -- 
        (\x+1.35,2.35) -- (\x+1.35,1.65) --
        (\x+0.35,1.65) -- (\x+0.35,0.65) --
        (\x-0.35,0.65) -- (\x-0.35,1.65) --
        (\x-1.35,1.65);
    }

    \node[draw,circle,fill=white] at (10,2) {$c$};
    \node[circle,draw,fill=white] at (10-1,2) {};
    \node[circle,draw,fill=white] at (10,3) {};
    \node[circle,draw,fill=white] at (10,1) {};
    \draw[thick] (10-1.35,1.65) -- (10-1.35,2.35) --
    (10-0.35,2.35) -- (10-0.35,3.35) -- (10+0.35,3.35) --
    (10+0.35,0.65) -- (10-0.35,0.65) --
    (10-0.35,1.65) -- (10-1.35,1.65);
    
    \end{tikzpicture}
    \caption{Cops placed on every other column can clear two rows.}
    \label{fig:interlock}
\end{figure}
As the cops march upwards, two rows are cleared, and only one row is reinfected each round. This strategy avoids overlapping the neighborhoods of the cops' probes, but only the ``forward'' edge of the cops' line clears infected vertices. The ``rear'' of each probe is already cleared. Matching the isoperimetric lower bound requires most cops to clear the maximum $\Delta(G)+1 = 5$ new vertices with every probe. 

To improve on this second strategy, note that it would take the robber multiple rounds to cross the cops' formation: three rounds on a column with a cop and two rounds on a column without a cop. This implies that cops only need to play on these positions every other round to prevent the robber from reinfecting the portion of the grid they've fully cleared. By shifting the set of columns where the cops' probe, it is possible for the cops to protect these rows by playing twice every five rounds, so that each column has two vertices cleared by one cop move and three vertices cleared by the other. 

We now give a high-level description of our strategy. We partition the grid into five rectangles of width approximately $n/5$ vertices. Using two sets of approximately $n/10$ cops, each rectangle is probed twice every five rounds with a cop probing every other column in the rectangle. To prevent the robber from slipping from the infected portion of one rectangle to the cleared portion of the next, the cops clear a diagonal, rather than the two rows from Figure~\ref{fig:interlock}. Although it takes several rounds to do so, we show that the cops can move these diagonals up the rectangles, closing in on the robber until they are captured.

We now turn to the proof of our main result in this section. 

\begin{proof}[Proof of Theorem~\ref{thm:grid_upper}]
Let $G_{n,m'}$ denote the $n \times m'$ grid and $G$ denote the square lattice with vertices in $\mathbb{Z}\times \mathbb{Z}$. 
It will be convenient to allow the cops to play on $G$ and to restrict the robber's position to a subgraph $G_{n,m'}$, where $m'=m$ for the first part of the proof and $m' = n$ in the second part of the proof. (Recall that $m$ is the odd integer for which $n = 5m-i$ for some $0 \le i \le 9$.) We will then argue this relaxation did not benefit the cop player. 

We break the proof into three parts which prove the following three claims, respectively: 
\begin{enumerate}
    \item When the robber is restricted to a subgraph $G_{n,m}$ within $G$, $\frac{m+3}{2}$ cops which only play on rounds $t$ with $t\equiv 0,3 \pmod{5}$ can capture the robber. 
    \item When the robber is restricted to a subgraph $G_{n,5m}$ of $G$, the robber can be captured by $m+3$ cops.
    \item At most $m+3$ cops are required to capture the robber on $G_{n,n}$. 
\end{enumerate} 

Throughout the proof, let 
\[
f_{i,j}(c) = 
\begin{cases}
        i+1+\lfloor \frac{c-j}{2} \rfloor & \text{for } c-j> 0\\ 
        i+\lceil \frac{c-j}{2} \rceil & \text{for } c-j \leq 0,
        \end{cases} 
\]
and define $F_{i,j}=\{(r,c) : 1 \leq c \leq m', f_{i,j}(c) \leq r \leq n\}$, which we call a \emph{forced region}. These forced regions describe different subsets of vertices that the cop player will contain the robber within. The strategy we describe will reduce the cardinality of the forced region over time.
It is important to note here that when we show that the robber must be in a given forced region, there may be some vertices the robber cannot occupy. In particular, it will be convenient to assume that the forced region will, in some rounds, contain several vertices $(x,y)$ with $x <1$.

The $(i,j)$ index of $F_{i,j}$ refers to a cell along the lower edge of the region, which cuts diagonally from southwest to northeast across columns $1$ through $m'$. In particular, the second index describes the \emph{column of focus} of the forced region, which we pay special attention to in the proof. The function $f_{i,j}$, given a column $c$, gives a certain key position in $F_{i,j}$ related to where the cops will play. See Figure~\ref{fig:grids_sweep_one} for a visual reference.

Let 
\begin{align*}
    S_{i,j} =& \{(f_{i,j}(c)+1,c) : c >j \text{ and } c-j \text{ is odd}\} \cup\\
    & \{(f_{i,j}(c)+1,c) : c \leq j \text{ and } c-j \text{ is even}\} .
\end{align*}

The set $S_{i,j}$ contains points spaced apart in an L-shape, similar to a knight move in chess. These vertices are chosen so that the neighborhoods of the points cover a diagonal stripe along the bottom of $F_{i,j}$ with minimal overlap. 

\smallskip

\noindent \emph{Claim 1}:  When the robber is restricted to a subgraph $G_{n,m}$ within $G$, $\frac{m+3}{2}$ cops which only play on rounds $t$ with $t\equiv 0,3 \pmod{5}$ can capture the robber. 

\smallskip

The cops' strategy consists of two moves. If the cops have restricted the robber to $F_{i,j}$ immediately before the cops' move, then the cops will play on the vertices $S_{i,j}$ that are within the columns $[0,m+1]$. By the definition of $S_{i,j}$, the cops will only be playing on every other column in $[0,j]$ and every other column in $[j+1,m+1]$, for a total of
\[ \left\lceil\frac{j}{2}\right\rceil + \left\lceil\frac{(m+1)-(j+1)+1}{2}\right\rceil \le \frac{j+1}{2} + \frac{(m+1)-(j+1)+1+1}{2} = \frac{m+3}{2} \]
cops. 

Immediately after this cop's move, the robber may only be on a vertex in $F_{i,j} \setminus N[S_{i,j}]=F_{i+2,j-1}$. 
To see this fact, we consider the cases $c>j$ and $c \leq j$. For the first case, when $c-j$ is odd, $N[(f_{i,j}(c)+1,c)]$ is a superset of $\{(f_{i,j}(c),c),(f_{i,j}(c)+1,c),(f_{i,j}(c)+2,c)\}$; if the robber is in column $c$, then it must be on a vertex $(r,c)$ with $r\geq f_{i,j}(c)+3=f_{i+2,j-1}(c)$. When $c-j$ is even, $\{(f_{i,j}(c),c),(f_{i,j}(c)+1,c)\}\subseteq N[(f_{i,j}(c)+1,c)]$, and the robber must be in a row $r\geq f_{i,j}(c)+2=f_{i+2,j-1}(c)$. The case when $c \leq j$ follows similarly to find that if the robber is in column $c$, then the robber must be on row $r\geq f_{i+2,j-1}(c)$, which defines the set $F_{i+2,j-1}$. 
We call this the \emph{natural} cop move and label this as P1 for future reference.

In the case that the robber is known to be on a vertex of $F_{i+2, 0}$, the cops instead think of the robber as being on a vertex of $F_{i+2+\frac{m+1}{2}, m}$. These two forced regions describe the same set of vertices, but as $S_{i,j}$ is determined by the index $(i,j)$, this change of index describes a different cop move. We label this replacement as P2 for future reference. 

Observe that when the robber is on a vertex in $F_{i,j}$ but the cops do not play during their next move (as in rounds $1,2,4 \pmod{5}$), then the robber moves to a vertex in $F_{i-1,j}=N[F_{i,j}]$. 
We label this as P3 for future reference. 

We prove Claim 1 recursively. 
As a base case, we can assume the robber is contained within $F_{i, m}$ for any $i \leq 1$, where we note that we will take $i$ to be negative in some cases.  
Note that for such $i$, $F_{i, m}$ contains the subset of vertices $[1,n]\times[1,m]$, so this initial assumption is always true.

For the recursive step, assume the robber is contained within some $F_{i, m}$ in round $t$ with $t \equiv 0\pmod{5}$. 
Repeating the natural move P1, the column of focus $j$ will shift from $m$ down to $1$, with rounds $0\pmod{5}$ focusing on an odd column and rounds $3\pmod{5}$ focusing on an even column. We refer to these rounds collectively as the first sweep. 
Once the column of focus is $1$, we continue to play, and the column of focus again becomes $m$. 
For this second sweep, the column of focus will shift from $m$ all the way down to $1$, but  with the parity reversed: rounds $3\pmod{5}$ focus on an odd column and rounds $0\pmod{5}$ focus on an even column. After both sweeps, the cops will have successfully moved the robber from $F_{i,m}$ to $F_{i+1,m}$.

We refer the reader to Figure~\ref{fig:grids_sweep_one}, which depicts the following cop moves.

\smallskip
{\bf First sweep:} The cops play the natural move so that the robber must be contained within $F_{i+2,m-1}$ (see P1). 
The robber moves three times, first to a vertex in $F_{i+1,m-1}$, then to a vertex in $F_{i,m-1}$, and finally to a vertex in $F_{i-1,m-1}$ (see P3). 
The cops then play the natural move in round $t+3$, so that the robber must be contained within $F_{i+1,m-2}$ (see P1). 
Then the robber moves twice, first to a vertex in $F_{i,m-2}$, then to a vertex in $F_{i-1,m-2}$ (see P3). 
This shows that every five moves, the indices of the forced region decrease by one in the first coordinate and two in the second coordinate. This process repeats $(m-1)/2$ times until the robber is known to reside on $F_{i-\frac{m-1}{2},1}$ immediately before the cops' $(t+5\frac{m-1}{2})$th move, where we note that  $t+5\frac{m-1}{2} \equiv 0 \pmod{5}$. 
The cops again play the natural move, and so the robber is known to reside in $F_{i-\frac{m-1}{2}+2,0}$, immediately after the cops' move, which is then replaced with $F_{i+3,m}$ by P2. 
The robber takes three moves and is on a vertex in $F_{i,m}$.  

\begin{figure}
    \centering
    \includegraphics[width=10cm]{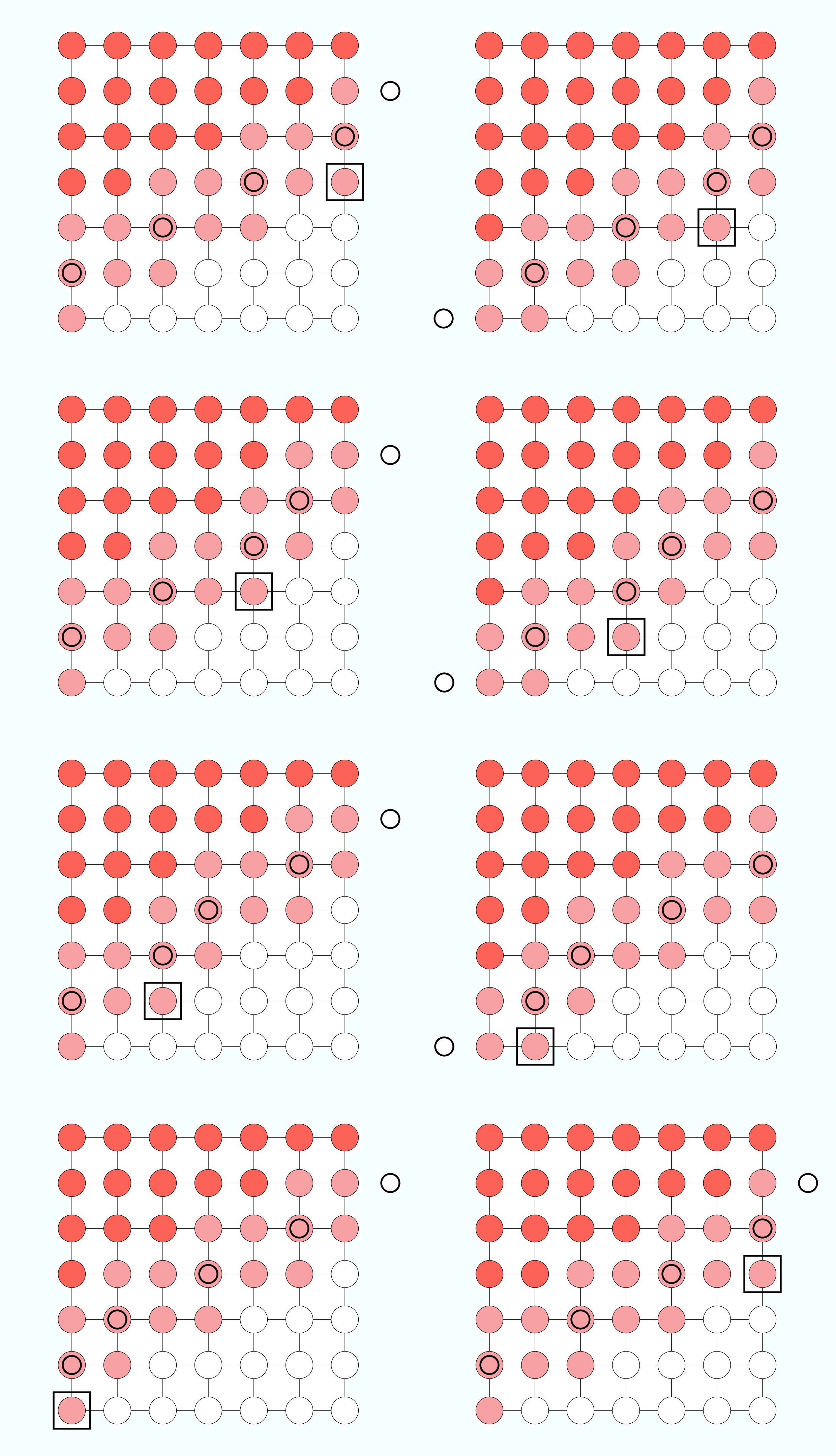}
    \caption{The first sweep of claim 1 where $m=7$. Red and pink dots indicate the forced region before the cops' move, and the square indicates the $i,j$ such that this forced region is $F_{i,j}$. 
    Red dots indicate the forced region after the cops' move. Black circles indicate the locations of the cops their move. From left to right, top to bottom, the images indicate play just before and after the cops' 1st, 4th, 6th, 9th, 11th, 14th, 16th, and 19th move.}
    \label{fig:grids_sweep_one}
\end{figure}

\smallskip
{\bf Second sweep:} The cops again play the natural move, and so the robber is known to reside in $F_{i+2,m-1}$. 
The robber takes two moves and is on a vertex in $F_{i,m-1}$. Once again, every five moves the first index of the forced region is decreased by one and the second index by two.
Play continues in this fashion (the cops playing the natural move on round $0,3\pmod{5}$) for $(m-1)/2$ rounds, until the robber is known to reside in $F_{i-\frac{m-1}{2},1}$ immediately before the cops' move in the $(t+5(m-1))$th round, where $t+5(m-1) \equiv 0 \pmod{5}$. 
The cops play the natural move, so the robber is known to reside in $F_{i-\frac{m-1}{2}+2,0}$ immediately after the cops' move, which is then replaced with $F_{i+3,m}$ by P2. 
The robber takes two moves and is on a vertex in $F_{i+1,m}$ immediately before the cops' $(t+5m)$th move.

We note that we started with the robber being on a vertex of $F_{i,m}$ in some round $0 \pmod{5}$, and now have that $F_{i+1,m}$ in some round $0 \pmod{5}$. This completes the recursive step, and so we can conclude that after sufficiently many rounds, as $F_{n+\frac{3m}{2},m}~=~\varnothing$,
the robber is captured, and the proof of Claim~1 follows.

For the next claim, it will be useful to note that if this process is initialized with $F_{i,m}$, then immediately before the cops' move in round $t=5m\alpha+1$, the robber must be in the forced region $F_{i+\alpha,m}$. 

\smallskip

\noindent \emph{Claim 2}:   When the robber is restricted to a subgraph $G_{n,5m}$ of $G$, the robber can be captured by $m+3$ cops. 

\smallskip

We split the subgraph $G_{n,5m}$ into five subgraphs $A_1, A_2,A_3,A_4,$ and $A_5$, where $A_j$ is on the vertices of $G_{n,5m}$ in columns $[(j-1)m+1, jm]$. 
We use a set of $\frac{m+3}{2}$ cops on $A_j$ on rounds $j,j+3 \pmod{3}$. 
This requires two sets of $\frac{m+3}{2}$ cops, and so $m+3$ cops are used in total. 
The technique described in Claim $1$ is applied to each $A_i$, shifting the rounds on which the cops play appropriately, and with the additional condition that the process in Claim 1 is started in $A_j$ in round $j$ with forced region $F_{i,m}$ where $i=-2m + (j-1)\frac{m-1}{2}$.

To illustrate how the process in Claim~1 is extended to $G_{n,5m}$, we describe the first six cops moves of the $m+3$ cops on the infinite square grid $G$, focused on the subgraph $G_{n,5m}$. 

\smallskip

\noindent \emph{Cop move 1}: The first set of $\frac{m+3}{2}$ cops play on $S_{-2m,m}$ on the columns $[0,m+1]$ (this is the first move for the cops in $A_1$).

\smallskip

\noindent \emph{Cop move 2}: The first set of $\frac{m+3}{2}$ cops play on $S_{-2m+\frac{m-1}{2},m}$ on the columns $[m,2m+1]$ (this is the first move for the cops in $A_2$). 

\smallskip

\noindent \emph{Cop move 3}: The first set of $\frac{m+3}{2}$ cops play on $S_{-2m+2\frac{m-1}{2},m}$ on the columns $[2m,3m+1]$ (this is the first move for the cops in $A_3$). 

\smallskip

\noindent \emph{Cop move 4}: The first set of $\frac{m+3}{2}$ cops play on $S_{-2m+3\frac{m-1}{2}, m}$ on the columns $[3m,4m+1]$ (this is the first move for cops in $A_4$) and the second set of $\frac{m+3}{2}$ cops play on $S_{-2m-1, m-1}$ on the columns $[0,m+1]$ (this is the fourth round in $A_1$).

\smallskip

\noindent \emph{Cop move 5}: The first set of $\frac{m+3}{2}$ cops play on $S_{-2m+4\frac{m-1}{2}, m}$ on the columns $[4m,5m+1]$ (this is the first move for cops in $A_5$) and the second set of $\frac{m+3}{2}$ cops play on $S_{-2m-1+\frac{m-1}{2},m-1}$ on the columns $[m,2m+1]$ (this is the fourth round in $A_2$).

\smallskip

\noindent \emph{Cop move 6}: The first set of $\frac{m+3}{2}$ cops play on $S_{-2m-1,m-2}$ on the columns $[0,m+1]$ (this is the sixth round in $A_1$) and the second set of $\frac{m+3}{2}$ cops play on $S_{-2m-1+2\frac{m-1}{2},m-1}$ on the columns $[2m,3m+1]$ (this is the fourth round in $A_3$).

\smallskip

If the robber stays within a subgraph $A_j$, then by Claim~1 they will eventually be captured.
Suppose the robber moves from one subgraph to another, say from $A_{a}$ to $A_b$. 
The robber must have been on a vertex of the cops' current forced region in $A_{a}$ in round $t$. 
If the robber moved to the cops' current forced region in $A_b$, then the robber has not made progress as they may as well have started in $A_b$ and stayed there until the current move. 
Therefore, we can assume that the robber moves to a vertex of $A_b$ outside of the cops' forced region. 

Assume without loss of generality that $a,b\in\{1,2\}$. 
We analyze the moves that occur on the border of $A_1$ and $A_2$, which affect where the robber can be on either column $m$ or $m+1$. 
Before we begin, we analyze each of the cops' moves of $A_1$ to find which cops played on either the left-most columns $0$, $1$, and $2$, or the right-most columns $m-1$, $m$, and $m+1$. 
This will require a deeper analysis of the moves in Claim~1. 
We note the following properties. 

\smallskip
{\bf Property 1}: The first sweep of Claim~1 utilized $5\frac{m-1}{2}+3$ rounds and the second sweep of Claim~1 utilized $5\frac{m-1}{2}+2$ rounds. Together, this is $5\frac{m-1}{2}+3+5\frac{m-1}{2}+2 = 5m$ rounds needed to perform both sweeps. Therefore, if the cops were on $S_{i,m}$ in round $t$, then the cops are on $S_{i+1,m}$ in round $t+5m$. Since the cops play on $S_{-2m,m}$ in the round with $t=1$, we conclude that the cops play on $S_{-2m+\alpha,m}$ during round $t=1+(5m)\alpha$. 

\smallskip
{\bf Property 2}: If the cops played on $S_{i,j}$ in round $t$, then the cops play on $S_{i-1,j-2}$ in round $t+5$, unless $j\in \{1,2\}$, in which case the cops play on $S_{i+\frac{m-1}{2},j+m-2}$.
Since the cops play on $S_{-2m+\alpha,m}$ in round $t=1+(5m)\alpha$, we conclude that the cops play on $S_{-2m+\alpha-\beta,m-2\beta}$ in round $t=1+(5m)\alpha+5\beta$ when $0 \leq \beta \leq \frac{m-1}{2}$, and the cops play on $S_{-2m+\alpha-\beta+\frac{m-1}{2},m-2\beta+(m-2)}$ in round $t=1+(5m)\alpha+5\beta$ when $\frac{m+1}{2} \leq \beta \leq m-1$. 

\smallskip
{\bf Property 3}: If the cops play on $S_{i,j}$ in round $t$ where $t\equiv 1 \pmod{5}$, then in the round $t+3$ the cops play on $S_{i-1,j-1}$ if $j \neq 1$, and on $S_{i+\frac{m-1}{2},j+(m-1)}$ if $j=1$. 

We next consider the situation where the cops play near the left and right edges. For each $t\equiv 1 \pmod{5}$, we describe which of these cops in $A_1$ played on a column in $\{0,1,2, m-1,m,m+1\}$. 
\begin{enumerate}
\item If we are playing in round $t=1+(5m)\alpha+5\beta$ where $\beta=0$, then the cops in these columns were played on vertices $\{(-2m+\alpha+1,m),(-2m+\alpha+2,m+1), (-2m+\alpha+1-\frac{m-1}{2},1)\}$.
\item If we are playing in round $t=1+(5m)\alpha+5\beta$ where $1 \leq \beta \leq \frac{m-1}{2}$, then the cops in these columns were played on vertices $\{(-2m+\alpha+1,m-1),(-2m+\alpha+2,m+1), (-2m+\alpha+1-\frac{m-1}{2},1)\}$.
\item If we are playing in round $t=1+(5m)\alpha+5\beta$ where $\frac{m+1}{2} \leq \beta \leq m-1$, then the cops in these columns were played on vertices $\{(-2m+\alpha+2,m),(-2m+\alpha+1-\frac{m-1}{2},0),(-2m+\alpha+2-\frac{m-1}{2},2)
\}$.
\end{enumerate}

For each $t\equiv 4 \pmod{5}$, we describe which of these cops in $A_1$ played on a column in $\{0,1,2, m-1,m,m+1\}$.  
\begin{enumerate}
\item If we are playing in round $t=4+(5m)\alpha+5\beta$ where $0 \leq \beta \leq \frac{m-1}{2}$, then the cops in these columns were played on vertices $\{(-2m+\alpha+1,m), (-2m+\alpha-\frac{m-1}{2},0), (-2m+\alpha+1-\frac{m-1}{2},2)\}$.  
\item If we are playing in round $t=4+(5m)\alpha+5\beta$ where $\beta = \frac{m+1}{2}$, then the cops in these columns were played on vertices $\{(-2m+\alpha+1,m), (-2m+\alpha+2,m+1),
(-2m+\alpha+1-\frac{m-1}{2},1)
\}$.
\item If we are playing in round $t=4+(5m)\alpha+5\beta$ where $\frac{m+1}{2} \leq \beta \leq m-1$, then the cops in these columns were played on vertices $\{(-2m-1+\alpha+ 1,m-1), (-2m+\alpha+2,m+1),
(-2m+\alpha+1-\frac{m-1}{2},1)
\}$. 
\end{enumerate}

\smallskip
{\bf Property 4}: A cop on $A_1$ plays on the vertex $(i,j)$ during round $t$ if and only if a cop on $A_2$ plays on the vertex $(i+\frac{m-1}{2},j+m)$ in round $t+1$.
As a consequence, for each of the vertices $(i,j)$ with $j \in \{0,1,2\}$ 
that were visited by a cop in $A_1$ in round $t$ as described above, the corresponding vertex $(i',j') = (i+\frac{m-1}{2},j+m)$ in $A_2$ was visited in round $t+1$, where  $j' \in \{m,m+1,m+2\}$. 
Therefore, for every round, we can now derive which cops probed a vertex in column $\{m-1,m,m+1,m+2\}$.

This is relevant as only the cops playing in columns $\{m-1,m,m+1,m+2\}$ will impact the robber's location on the border of $A_1$ and $A_2$. To simplify, take $\alpha=2m$. A similar argument follows for all other $\alpha$. 
In Table~\ref{tab:cop_placements_part2}, we describe exactly which vertices are probed by the cops on columns $\{m-1,m,m+1,m+2\}$ in rounds $i + 5\beta + 5m(2m)$, where $0 \leq \beta \leq m-1$ and $1 \leq i \leq 5$.

\begin{table}[h]
    \centering
\begin{tabular}{ c || c c c c c | c c c c c }
\hline
 i& $\beta=0$& $1 \leq \beta \leq \frac{m-1}{2}$& $\beta=\frac{m+1}{2}$ & $\frac{m+1}{2} \leq \beta \leq m-1$  \\
\hline
$1$ & 
$(1,m)$ $(2,m+1)$
&$(1,m-1)$ $(2,m+1)$
&$(2,m)$
&$(2,m)$\\
$2$ & $(1,m+1)$
&$(1,m+1)$
&$(1,m), (2,m+2)$ 
&$(1,m)$ $(2,m+2)$\\
$3$&&&&\\
$4$ & $(1,m)$&$(1,m)$&$(1,m)$ $(2,m+1)$&$(1,m-1)$ $(2,m+1)$\\ 
$5$ & $(0,m)$ $(1,m+2)$&$(0,m)$ $(1,m+2)$&$(1,m+1)$&$(1,m+1)$ 
\end{tabular}
    \caption{The vertices in columns $\{m-1,m,m+1,m+2\}$ where a cop plays in round $t=i+5\beta+5m\alpha$, where $\alpha=2m$.}
    \label{tab:cop_placements_part2}
\end{table}

We may also analyze Claim~1 to find that if the robber is on a vertex $(i,m)$ in column $m$ that is in the forced region of $A_1$ immediately after the cops move in round $t=1+5\beta + (5m) \alpha$, then $i \geq \alpha-2m+3$ if $0 \leq \beta \leq \frac{m-1}{2}$, and $i \geq \alpha-2m+4$ if $\frac{m+1}{2} \leq \beta \leq m-1$. Similarly, after the cops move in round $t=4+5\beta + (5m) \alpha$, then $i \geq \alpha-2m+3$ for $0 \leq \beta \leq m-1$. 

If the robber is on a vertex $(i,m+1)$ in column $m+1$ that is in the forced region of $A_2$ immediately after the cops move in round $t=2+5\beta + (5m) \alpha$, then $i \geq \alpha-2m+3$. 
Similarly, after the cops move in round $t=5+5\beta + (5m) \alpha$, then $i \geq \alpha-2m+2$ if $0 \leq \beta \leq \frac{m-1}{2}$, and $i \geq \alpha-2m+3$ if $\frac{m+1}{2} \leq \beta \leq m-1$.
 
Let $x_1^t$ denote the smallest value of $x$ such that a cop may be on $(x,m)$ in round $t$ in the forced region of $A_1$, and let $x_2^t$ denote the smallest value of $x$ such that a cop may be on $(x,m+1)$ in round $t$ in the forced region of $A_2$. 
The robber may then move from the forced region of $A_1$ onto a vertex not in the forced region of $A_2$ only when $x_1^t \leq x_2^t+2$. Similarly, the robber may move from the forced region of $A_2$ onto a vertex not in the forced region of $A_1$ only when $x_2^t \leq x_1^t+2$. We note that by the above analysis, this only occurs when $t \equiv 1,4 \pmod{5}$. 
In each of these rounds and for every possible move of the robber from a vertex of a forced region onto a vertex not in a forced region, there is a cop that prevents it by either being adjacent to the robber before or after their move. The complete list of such events is presented in Table~\ref{tab:cops_on_border_patrol} for the case $\alpha=2m$. 
The proof of Claim~2 follows.

\begin{table}[h]
    \centering
\begin{tabular}{ c c || c c c c c | c c c c c }
\hline
&$t\equiv$ 
& robber at $t$
& robber at $t+1$
& capturing cop\\
\hline
$0 \leq \beta\leq \frac{m-1}{2}$ &1&(1,$m$+1)&(1,$m$) & $(2,m+1)$ on round $t$\\
 &4&(1,$m+1$)&(1,$m$) & $(1,m)$ on round $t$\\
\hline
$\beta =\frac{m+1}{2}$ &1&(1,$m+1$)&(1,$m$) & $(1,m)$ on round $t+1$\\
 &1&(2,$m$+1)&(2,$m$) & $(1,m)$ on round $t+1$\\
 &4&(1,$m+1$)&(1,$m$) & $(1,m)$ on round $t$\\
\hline
$\frac{m+3}{2} \leq \beta\leq m-1$ &1&(2,$m+1$)&(2,$m$) & $(2,m)$ on round $t$\\
 &4&(2,$m+1$)&(1,$m+1$) & $(2,m+1)$ on round $t$\\
\hline
\end{tabular}
    \caption{
    In round $t=i+5\beta+(5m)\alpha$ with $\alpha=2m$, each possible robber move from the forced region of one $A_j$ to the unforced region of the other $A_{j'}$ is represented as a row, with the corresponding cop that captures the robber if it performs this move.}\label{tab:cops_on_border_patrol}
\end{table}

\smallskip

\noindent \emph{Claim 3}:  At most $m+3$ cops are required to capture the robber on $G_{n,n}$.

\smallskip

We now show that at most $m+3$ cops are required to capture the robber on $G_{n,n},$ which proves Claim~3 and will complete the proof of the theorem. Recall that $n=5m-i$ for some 
$0\le i \le 9$. 
To capture the robber on $G_{n,n}$, the cops' will observe and modify the strategy to capture the robber on the subgraph $G_{n,5m}$ of $G$ given in Claim~2. 
That is, suppose that the $m+3$ cops play on vertices $S_t$ in round $t$ in $G$ where the robber is restricted to $G_{n,5m}$. 
We further restrict the robber so that it can only be on $G_{n,n}\subseteq G_{n,5m}$. 
A simple modification to the cop moves $S_t$ also ensures that the cops only play on the subset $G_{n,n}$ of $G$. However, this game is identical to just playing on the graph $G_{n,n}$, and so is a winning strategy for $m+3$ cops to capture the robber on $G_{n,n}$. 

We note that each cop outside of $[0,n+1]\times[0,n+1]$ will not affect the robber, since the robber is contained within the vertices $[1,n]\times[1,n]$ in all rounds. Delete all vertices in $S_t$ that are not within $[0,n+1]\times[0,n+1]$. This has no impact on capturing the robber. 

Suppose $(0,x)\in S_t$. This cop clears only the vertex $(1,x)$ in round $t$, and so it is a strictly better move for the cop to play on $(1,x)$. We therefore, replace $(0,x)\in S_t$ with $(1,x)\in S_t$. Similarly, we replace $(n+1,x)\in S_t$ with $(n,x)\in S_t$, replace $(x,0)\in S_t$ with $(x,1)\in S_t$, and replace $(x, n+1)\in S_t$ with $(x,n)\in S_t$. 

The resulting cop moves are, therefore, strictly better at capturing the robber on $G_{n,n}$, but also have the robber contained within $G_{n,n}$. This completes the proof. 
\end{proof}

\section{Conclusion and future directions}

We introduced the one-visibility localization number and proved asymptotically tight bounds on Cartesian grids and bounds on $k$-ary trees. We gave bounds for trees in terms of their order and depth. Determining a tree's exact one-visibility localization number based on its structural features remains an open problem. 

The one-visibility localization number may be investigated in various graph families where the localization number has been studied, such as Kneser graphs, Latin square graphs, or the incidence graphs of projective planes and combinatorial designs. Our approach using isoperimetric inequalities should apply to the families of hypercubes, higher dimensional Cartesian grids, and strong grids.

Another natural direction would be to consider the $k$-limited visibility Localization game for $k > 1,$ with corresponding optimization parameter $\zeta_k$. It would be interesting to find graphs $G$ such $\zeta_i(G) \ne \zeta_j(G)$ for all distinct values of $i$ and $j$ that are at most the radius of $G.$

\section{Acknowledgements}
The authors were supported by NSERC.

\end{document}